\documentclass[11pt]{article}

\usepackage{amssymb,amsmath}
\usepackage{amsthm}
\usepackage{thmtools} 

\usepackage{cite}
\usepackage{graphicx}

\usepackage{enumerate}
\usepackage{algorithm,algpseudocode}

\usepackage[colorlinks=true,allcolors=blue]{hyperref}

\usepackage{caption}
\newcommand{\ds}{\displaystyle}
\newcommand{\ep}{\varepsilon}
\newcommand{\RR}{\mathbb{R}}
\newcommand{\lev}{\operatorname{lev}}
\newcommand{\prox}{\operatorname{prox}}

\newcommand{\Argmin}{\operatornamewithlimits{Argmin}}
\newcommand{\argmin}{\operatornamewithlimits{argmin}}

\newcommand{\dist}{\operatorname{dist}}
\newcommand{\E}{\mathbb{E}}

\newcommand{\norm}[1]{\left\|#1\right\|}
\newcommand{\innprod}[2]{\left\langle#1,#2\right\rangle}

\newcommand{\APGIter}{\operatorname{\tt APGIter}}
\newcommand{\PGIter}{\operatorname{\tt PGIter}}
\newcommand{\AdaAPG}{\operatorname{\tt AdaAPG}}
\newcommand{\rAdaAPG}{\operatorname{\tt rAdaAPG}}

\newcommand{\deltaf}{\delta_{\rm F}}
\newcommand{\deltag}{\delta_{\rm G}}

\declaretheorem[numberwithin=section]{theorem}
\declaretheorem[numberlike=theorem]{proposition}
\declaretheorem[numberlike=theorem]{corollary}
\declaretheorem[numberlike=theorem]{lemma}
\declaretheorem[style=definition,numberlike=theorem]{definition}
\declaretheorem[style=remark,qed=\qedsymbol,numberlike=theorem]{remark}

\title{Nearly optimal first-order methods for convex optimization under gradient norm measure: An adaptive regularization approach}
\author{Masaru Ito%
	\thanks{
		Department of Mathematics, College of Science and Technology, Nihon University,
		1-8-14 Kanda-Surugadai, Chiyoda, Tokyo 101-8308, Japan 
		(\texttt{ito.m@math.cst.nihon-u.ac.jp)}.
	}
	\and
	Mituhiro Fukuda%
	\thanks{
		Department of Mathematical and Computing Science,
		Tokyo Institute of Technology,
		2-12-1-W8-41 Oh-okayama, Meguro, Tokyo 152-8552, Japan
		(\texttt{mituhiro@is.titech.ac.jp}).
	}
}
\date{October 2020}

\begin{document}
\maketitle


\begin{abstract}
In the development of first-order methods for smooth (resp., composite) convex optimization problems, where smooth functions with Lipschitz continuous gradients are minimized, the gradient (resp., gradient mapping) norm becomes a fundamental optimality measure. Under this measure, a fixed iteration algorithm with the optimal iteration complexity is known, while determining this number of iteration to obtain a desired accuracy requires the prior knowledge of the distance from the initial point to the optimal solution set.
In this paper, we report an adaptive regularization approach, which attains the nearly optimal iteration complexity without knowing the distance to the optimal solution set.
To obtain further faster convergence adaptively, we secondly apply this approach to construct a first-order method that is adaptive to the H\"olderian error bound condition (or equivalently, the {\L}ojasiewicz gradient property), which covers moderately wide classes of applications.
The proposed method attains nearly optimal iteration complexity with respect to the gradient mapping norm.
\\

\noindent
\textbf{Keywords:} smooth/composite convex optimization; accelerated proximal gradient methods; H\"olderian error bound; adaptive methods
\\

\noindent
\textbf{Mathematical Subject Classification (2010):} 90C25; 68Q25; 49M37
\end{abstract}


\section{Introduction}

The class of proximal gradient methods is a fundamental tool
for the composite convex optimization problems, which consists in minimizing the sum of a differentiable convex function with Lipschitz continuous gradient and a (possibly non-smooth) convex function with a ``simple'' structure (called the regularizer).
Accelerated first-order methods for this class of problems have been well-studied as
they provide optimal iteration complexity to obtain an approximate solution under the measure of the \emph{functional residual}, i.e., the gap between the objective function value at an iterate and the optimal value, for various classes of problems \cite{BT09,Nes83,Nes13,Nes15,RG18,Rou17}.

A major interest we focus on in this paper is to consider the iteration complexity under the measure of the norm of the \emph{gradient mapping}. The gradient mapping generalizes the gradient in the smooth case (i.e., the case when there is no regularizer) to the composite case,
and the gradient mapping vanishes at a point if and only if that point is an optimal solution.
The norm of the gradient mapping is a useful optimality measure as it is computable at each iteration (if the proximal mapping of the regularizer is computable), while the functional residual itself is not verifiable if we do not know the optimal value. Remark that, in some cases, we are able to provide computable upper bounds on the functional residual, such as the duality gap \cite{Nes05,Nes09,Lu09}, and the convergence is guaranteed for that bounds.

To find an approximate solution under the gradient norm measure, a first-order method with optimal iteration complexity for the smooth case was recently announced by Nesterov et al. (see \cite[Remark~2.1]{Nes20} and \cite[Remark~3.1]{KF20}).
This is a fixed iteration algorithm and, in order to determine the number of iterations to attain a desired accuracy, it requires the prior knowledge of the distance, say $D$, from the initial point to the optimal solution set.
Without the two requirements of fixing the number of iterations and knowing the distance $D$, variants of accelerated first-order methods (see, e.g., \cite{KF18,Nes12}) seem to be the best known ones, however their iteration complexity results become worse.
One aim of this paper is to provide an algorithm (Algorithm~\ref{alg:AdaAPG}), that attains the nearly optimal iteration complexity without these two requirements, as shown in Corollary~\ref{cor:AdaAPG}. This result is attained by an arrangement of the regularization technique \cite{Nes09}, which applies the accelerated first-order method to the regularized auxiliary problem adding a quadratic regularizer to the original objective function. Our algorithm ensures the same iteration complexity as the original regularization technique \cite{Nes09}, with an additional advantage that we do not need the distance $D$.

Another motivation in this paper is the development of an adaptive first-order method under the additional assumption, the \emph{H\"olderian Error Bound} (HEB) condition,
in order to provide better convergence guarantees as the linear convergence for strongly convex functions.
This condition is also related to the concept called the {\L}ojasiewicz gradient inequality \cite{Loj63,Loj65}, which is a useful tool for the development and the analysis of algorithms as well as first-order methods \cite{Att13,Bol17}.
These concepts are known to be satisfied under moderately mild assumptions such as when the objective function is coercive and subanalytic (in particular, semi-algebraic) \cite{Bol07}.
The HEB condition involves two parameters, the coefficient $\kappa$ and the exponent $\rho$, which are critical to determine the convergence rate, although they are difficult to estimate in general. Therefore, the development of adaptive algorithms is an important issue.

Under the HEB condition, we propose Algorithm~\ref{alg:rAdaAPG}, a restart scheme of the previous mentioned adaptive regularization algorithm.
Our method is inspired by the Liu-Yang's method \cite{LY17} as we employ an (approximate) proximal-point approach, where the main difference is the adaption parameter: Liu-Yang's method adaptively estimates the coefficient $\kappa$, but requires the prior knowledge of the exponent $\rho$, while our method adaptively determines the regularization parameter to define the regularized auxiliary problem.
As a result, without knowing the coefficient $\kappa$ and the exponent $\rho$, the proposed method adapts both the parameters $\kappa$ and $\rho$.
Under the measure of the gradient mapping norm, the proposed method ensures the same iteration complexity as the one by Liu and Yang \cite{LY17}, dismissing the prior knowledge of $\rho$ (Corollary~\ref{cor:rAdaAPG}).
In addition, we show the finite or the superlinear convergence when the exponent $\rho$ is less than two, which were not shown in the existing adaptive methods \cite{LY17,RG18}.
Furthermore, for the smooth case, we prove that the proposed method attains the nearly optimal iteration complexity with respect to the gradient norm. The same assertion can also be derived with respect to the measure of the functional residual.

Table~\ref{table:methods} shows the relation to existing adaptive first-order methods.
All the algorithms in this table attain the nearly optimal iteration complexity with respect to the employed measure.
The recent first-order methods \cite{RG18,Rou17} are applicable to our problem (for a specific regularizer) and they adapt both $\kappa$ and $\rho$ ensuring the nearly optimal iteration complexity under the measure of the functional residual.
One advantage of our method compared with these methods is the (nearly optimal) \emph{convergence} guarantee;
the method of Roulet and d'Aspremont (see \cite[Proposition 3.4]{Rou17}) is a fixed step algorithm (remark that, if we know the optimal value, \cite[Algorithm 3]{Rou17} gives nearly optimal convergence), and Renegar-Grimmer's method \cite{RG18} fixes the target tolerance.
Although the other algorithms, i.e., this work and the first four algorithms \cite{FQ17,LX15,LY17,Nes13} in Table~\ref{table:methods}, terminate if a desired approximate solution is found, they provide the nearly optimal convergence by omitting the stopping criterion as we discuss later.


This paper is organized as follows.
Section~\ref{sec:pre} collects preliminary facts on the gradient mapping and the HEB condition.
In particular, in Section~\ref{ssec:low-compl}, we deduce a lower iteration complexity bound with respect to the gradient norm for the class of smooth convex functions satisfying the HEB condition.
We review in Section~\ref{sec:APG} the regularization technique \cite{Nes12} preparing auxiliary results.
Two adaptive first-order methods and the main results for them are established in Section~\ref{sec:AdaAPG}.
In Section~\ref{ssec:AdaAPG}, we show an adaptive regularization algorithm and prove the nearly optimal iteration complexity for the class of smooth convex functions.
A restart scheme of this algorithm is given in Section~\ref{ssec:rAdaAPG} and we show that it adaptively attains the nearly optimal iteration complexity under the HEB condition. Concluding remarks are given in Section~\ref{sec:conclusion}.

\begin{table}[tbh]
\caption{\small
Adaptive first-order methods.
The column `Problem class' indicates the class of problems to which the algorithm is applicable.
The column `Optimality measures' is the optimality measure for which the nearly optimal iteration complexity was proved.
The column `Nearly optimal convergence' shows whether the algorithm ensures the convergence of the optimality measure to zero at the nearly optimal rate (while all the algorithms attain the nearly optimal iteration complexity for a fixed tolerance).
}
\label{table:methods}
\resizebox{\textwidth}{!}{
\begin{tabular}{l||c|c|c}
\hline
%
%
\qquad\quad Algorithm &
Problem class
&
Optimality measures &

Nearly optimal convergence
\\\hline\hline
%
%
Nesterov \cite{Nes13}
& $\mu$-strongly convex $\varphi$ with unknown $\mu$ & gradient mapping norm & yes \\
%
%
Lin and Xiao \cite{LX15} & & &\\\hline
%
%
Fercoq and Qu \cite{FQ17} & HEB with $\rho=2$ and unknown $\kappa$ & gradient mapping norm & yes \\\hline
%
%
Liu and Yang \cite{LY17} & HEB with known $\rho$ and unknown $\kappa$
& gradient mapping norm & yes \\\hline
%
%
Roulet and d'Aspremont & HEB with
unknown $\kappa$ and $\rho$
& functional residual
& only for predefined number of \\
\cite[Proposition 3.4]{Rou17} &&& iterations\\\hline
%
%
Reneger and Grimmer \cite{RG18}
& HEB with unknown $\kappa$ and $\rho$
& functional residual
& only for predefined tolerance \\
\hline
%
%
This work (Algorithm~\ref{alg:rAdaAPG}) & HEB with unknown $\kappa$ and $\rho$ & gradient mapping norm & yes \\
&& (and functional residual) \\\hline
\end{tabular}
}
\end{table}


\section{Preliminaries}\label{sec:pre}
Throughout this paper, let $\E$ be a finite dimensional real Hilbert space endowed with an inner product $\innprod{\cdot}{\cdot}$.
We denote by $\norm{x}=\innprod{x}{x}^{1/2}$ the induced norm on $\E$.
The distance between a point $x \in \E$ and a set $S \subset \E$ is defined by
$\dist(x,S)=\inf_{z \in S}\norm{z-x}.$

\sloppy 
We focus on the the convex composite optimization problem
\begin{equation}\label{main-prob}
\varphi^* = \min_{x \in \E}[\varphi(x) \equiv f(x)+\varPsi(x)],
\end{equation}
where $f:\E\to \RR$ is an \emph{$L_f$-smooth} convex function, that is, it is continuously differentiable and its gradient is $L_f$-Lipschitz continuous:
$$
\norm{\nabla{f}(x)-\nabla{f}(y)} \leq L_f\norm{x-y}, \quad \forall x,y \in \E,
$$
and ${\varPsi:\E\to\RR\cup\{+\infty\}}$ is a proper, lower-semicontinuous, and convex function.
We denote by ${X^*:=\Argmin_{x \in \E}\varphi(x)}$ the set of optimal solutions of $\min_{x \in \E}\varphi(x)$.
Throughout the paper, we assume that $X^*$ is nonempty.
The subdifferential of $\varphi$ at $x$ is denoted by
${\partial\varphi(x) = \{g \in \E \mid \varphi(y)\geq \varphi(x)+\innprod{g}{y-x},~\forall y \in \E\}}$.

For the objective function $\varphi(x)=f(x)+\varPsi(x)$, we define
$$
m_L(y;x):= f(y)+\innprod{\nabla{f}(y)}{x-y} + \frac{L}{2}\norm{x-y}^2+\varPsi(x),\qquad T_L(y):= \argmin_{x \in \E}m_L(y;x), 
$$
where the minimizer $T_L(y)$ is well-defined since $m_L(y;\cdot)$ is strongly convex.
It is assumed that $\varPsi(\cdot)$ has a ``simple'' structure, namely, the proximal mapping $\prox_{\varPsi}(y):=\argmin_{x \in \E}\{\varPsi(x)+\frac{1}{2}\norm{x-y}^2\}$ as well as the map $T_L(y)\equiv \prox_{\varPsi/L}(y-\nabla{f}(x)/L)$ is moderately computable (see \cite{PB14} for examples).
The \emph{gradient mapping} of $\varphi$ is defined by
\[
g_L(y):=L(y-T_L(y)),\quad y \in \E,~L>0.
\]
For instance, if $\varPsi\equiv 0$, we see that $T_L(y)=y-\nabla{f}(y)/L$ and $g_L(y)=\nabla{f}(y)$ hold.

Remark that the norm of $g_L(y)$ is given by
$
\norm{g_L(y)} = L\norm{y-T_L(y)},
$
from which the quantity $\norm{g_L(y)}$ can be used as a computable optimality measure at $y$ (see Lemma~\ref{lem:gmap} (ii) below).

The following lemma collects basic properties on $T_L(x)$ and $g_L(x)$, which can be found in \cite{BT09,Nes04,Nes13}.

\begin{lemma}\label{lem:gmap}
Let $\varphi=f+\varPsi$ be the sum of a continuously differentiable convex function $f:\E\to\RR$ and a proper lower-semicontinuous convex function $\varPsi:\E\to\RR\cup\{+\infty\}$.
Then, the following assertions hold.
\begin{enumerate}[(i)]
\item For all $y \in \E$, the map $L \mapsto \norm{g_L(y)}$ is increasing.
\item $x^* \in X^*$ holds if and only if $g_L(x^*)=0$.
\item $\nabla{f}(T_L(y))-\nabla{f}(y)+g_L(y) \in \partial\varphi(T_L(y))$ for all $y \in \E$. Moreover, if $f$ is $L_f$-smooth, then
$$
\norm{\nabla{f}(T_L(y))-\nabla{f}(y)+g_L(y)} \leq \left(\frac{L_f}{L}+1\right)\norm{g_L(y)}.
$$
\item If $y\in \E$ and $L>0$ satisfy $\varphi(T_L(y))\leq m_L(y;T_L(y))$, then we have the inequalities
$$
\frac{1}{2L}\norm{g_L(y)}^2 \leq \varphi(y)-\varphi(T_L(y)) \leq \varphi(y)-\varphi^*,
\qquad
\frac{1}{2L}\norm{g_L(y)} \leq \dist(y,X^*).
$$
\item If $f$ is $L_f$-smooth, then $\varphi(T_L(y))\leq m_L(y;T_L(y))$ holds for all $y\in \E$ and $L\geq L_f$.
\end{enumerate}
\end{lemma}
\begin{proof}
(i) is proved in \cite[Lemma 2]{Nes13}.

(ii) According to the optimality condition of the problem $\min_{x \in \E}m_L(y;x)$, we have the following equivalence for $y,z \in \E$:
\begin{equation}\label{optcond-TL}
z = T_L(y) \iff 0 \in \nabla{f}(y)+L(z-y)+\partial\varPsi(z).
\end{equation}
On the other hand, the optimality of the original problem $\min_{x}\varphi(x)$ is characterized as follows. Given $w\in \E$, we have
$w \in \Argmin_{x \in \E}\varphi(x)$ if and only if $0 \in \nabla{f}(w)+\partial{\varPsi}(w)$ holds
(Remark that we have $\partial\varphi=\nabla{f}+\partial \varPsi$ \cite[Theorem~23.8]{Roc}).
Let us show the equivalence
\begin{equation}\label{eq:gmap-opt-sol-ch}
x^* \in \Argmin_{x \in \E}\varphi(x) \iff x^*=T_L(x^*) ~~(\iff g_L(x^*)=0).
\end{equation}
If $x^*=T_L(x^*)$, then \eqref{optcond-TL} implies $0 \in \nabla{f}(x^*)+\partial \varPsi(x^*)$ and so $x^* \in \Argmin_{x\in \E}\varphi(x)$ holds. Conversely, if $x^*\in \Argmin_{x \in \E}\varphi(x)$, then we have the condition $0\in \nabla{f}(x^*)+\partial \varPsi(x^*)$, which coincides with the right hand one of \eqref{optcond-TL} in the case $y=z=x^*$. Thus, we conclude
 its equivalent counterpart $x^*=T_L(x^*)$. This proves the equivalence \eqref{eq:gmap-opt-sol-ch} as well as the assertion (ii).

(iii) By the condition \eqref{optcond-TL}, we have
$
0 \in \nabla{f}(y)+L(T_L(y)-y) + \partial\varPsi(T_L(y)) = \nabla{f}(y)-g_L(y)  + \partial\varPsi(T_L(y)).
$
Hence, we obtain
$
\nabla{f}(T_L(y))-\nabla{f}(y)+g_L(y) \in \nabla{f}(T_L(y))+\partial\varPsi(T_L(y)) = \partial\varphi(T_L(y))
$, giving the former assertion of (iii).
Now if $f$ is $L_f$-smooth, we have
$
\norm{\nabla{f}(T_L(y))-\nabla{f}(y)} \leq L_f\norm{T_L(y)-y} = \frac{L_f}{L}\norm{g_L(y)},
$
which yields the latter assertion of (iii).

(iv)
It is shown in \cite[Lemma~2.3]{BT09} that,
if $\varphi(T_L(y))\leq m_L(y;T_L(y))$ holds, then we have
$$
\varphi(x)  - \varphi(T_L(y)) \geq \frac{1}{2L}\norm{g_L(y)}^2 + \innprod{g_L(y)}{x-y},\quad \forall x \in \E.
$$
Letting $x:=y$ shows the first assertion. On the other hand, since $\varphi(T_L(y))\geq \varphi^*$, letting $x:=x^* \in X^*$ gives
$
\frac{1}{2L}\norm{g_L(y)}^2 \leq \innprod{g_L(y)}{y-x^*} \leq \norm{g_L(y)}\norm{y-x^*}.
$
Thus, we obtain $\frac{1}{2L}\norm{g_L(y)} \leq \dist(y,X^*)$.

(v) Note that $f$ is $L$-smooth for any $L\geq L_f$.
Since $L$-smooth functions satisfy (e.g., see \cite{Nes04})
$$
f(x) \leq f(y)+\innprod{\nabla{f}(y)}{x-y} + \frac{L}{2}\norm{x-y}^2,\quad \forall x,y \in \E,
$$
we obtain $\varphi(x) \leq m_L(y;x)$ for all $x,y \in \E$ and $L\geq L_f$.
\end{proof}

\subsection{H\"olderian Error Bound}

Here, we introduce the H\"olderian error bound condition, which is also discussed in the context of {\L}ojasiewicz inequality \cite{Bol07,Loj59,Loj65}.

\begin{definition}
Fix $x_0 \in \E$.
We say that $\varphi$ satisfies the \emph{H\"olderian error bound condition} with coefficient $\kappa > 0$ and exponent $\rho \geq 1$ if
\begin{equation}\label{HEB}
\varphi(x)-\varphi^* \geq \kappa\dist(x,X^*)^\rho,\quad \forall x\in \lev_{\varphi}(\varphi(x_0)),
\end{equation}
where $\lev_{\varphi}(\gamma)=\{x \in \E \mid \varphi(x)\leq \gamma\}$.
\end{definition}

According to \cite[Theorem~3.3]{Bol07}, the H\"olderian error bound condition is satisfied for some $\kappa$ and $\rho$ if $\varphi(x)$ is a proper, lower-semicontinuous, convex, coercive, and subanalytic function.
As subanalytic functions contain semi-algebraic ones, this condition appears in many applications including popular problems in machine learning; see, e.g., \cite{Bol17,LY17} for related studies.

A noteworthy concept related to the H\"olderian error bound condition is the \emph{{\L}ojasiewicz gradient inequality} \cite{Loj63,Loj65},
which is of the form
\begin{equation}\label{Loj-ineq-orig}
\dist(0,\partial\varphi(x)) \geq \lambda(\varphi(x)-\varphi^*)^\theta,\quad \forall x \in \lev_\varphi(\varphi(x_0))\setminus X^*,
\end{equation}
for $\lambda >0$ and $\theta \in [0,1[$.
In fact, these concepts \eqref{HEB} and \eqref{Loj-ineq-orig} are equivalent for convex functions (see Remark~\ref{rem:equiv-HEB-Loj}).
The {\L}ojasiewicz gradient inequality is a powerful tool for analyzing the convergence of first-order methods as it covers wide class of applications and algorithms \cite{Att13,Bol17}.

\begin{lemma}\label{lem:HEB}
Let $\varphi:\E\to\RR\cup\{+\infty\}$ be a proper, lower-semicontinuous, and convex function.
For $x_0 \in \E$, suppose that $\varphi$ satisfies the H\"olderian error bound condition \eqref{HEB} with coefficient $\kappa > 0$ and exponent $\rho \geq 1$.
Then, for every $x \in \lev_{\varphi}(\varphi(x_0))\setminus X^*$, we have
the inequalities
\begin{equation}\label{Loj-ineq}
\kappa \dist(x,X^*)^{\rho-1} \leq \inf\{\norm{g}:g \in \partial\varphi(x)\},\quad
\kappa^{\frac{1}{\rho}} (\varphi(x)-\varphi^*)^{\frac{\rho-1}{\rho}} \leq \inf\{\norm{g}:g \in \partial\varphi(x)\}.
\end{equation}
\end{lemma}
\begin{proof}
Let $x^*$ be the projection of $x$ onto $X^*$, so that $\norm{x-x^*}=\dist(x,X^*)$.
For every $g\in\partial \varphi(x)$, we have
\begin{align*}
\kappa \dist(x,X^*)^\rho
&\leq \varphi(x)-\varphi^*
\leq -\innprod{g}{x^*-x}
\leq \norm{g}\dist(x,X^*)\leq \norm{g}\frac{1}{\kappa^{1/\rho}}(\varphi(x)-\varphi^*)^{1/\rho}.
\end{align*}
It involves two inequalities
$
\kappa \dist(x,X^*)^\rho \leq \norm{g}\dist(x,X^*)$
and
$\varphi(x)-\varphi^* \leq \norm{g}\frac{1}{\kappa^{1/\rho}}(\varphi(x)-\varphi^*)^{1/\rho}.
$
Since $x \not \in X^*$, the assertion follows dividing them by $\dist(x,X^*)$ and $\kappa^{-1/\rho}(\varphi(x)-\varphi^*)^{1/\rho}$, respectively.
\end{proof}

\begin{remark}\label{rem:equiv-HEB-Loj}
The latter condition in \eqref{Loj-ineq} is the {\L}ojasiewicz gradient inequality \eqref{Loj-ineq-orig}
with the correspondence $\lambda=\kappa^{\frac{1}{\rho}}$ and $\theta = \frac{\rho-1}{\rho} \in [0,1[$.
It is shown in \cite{Bol17} that the {\L}ojasiewicz gradient inequality is essentially equivalent to the H\"olderian error bound condition: If \eqref{Loj-ineq-orig} holds with $\lambda=\rho\kappa^{\frac{1}{\rho}}$ and $\theta = \frac{\rho-1}{\rho}$, then \eqref{HEB} holds (In \cite[Theorem~5~(i)]{Bol17}, set $(\kappa^{-\frac{1}{\rho}} s^{\frac{1}{\rho}}, \varphi(x_0))$ in place of $(\varphi(s),r_0)$ and let the radius $\rho$ of $B(\bar{x},\rho)$ to $+\infty$).
\end{remark}

\subsection{Lower Complexity Bounds}\label{ssec:low-compl}
Let us discuss lower bounds on the iteration complexity under the H\"olderian error bound condition with respect to the optimality measure $\norm{g_L(x)}$. The lower bound is derived in the case $\varPsi\equiv 0$ so that $\varphi=f$ is a smooth function and we have $g_L(x)=\nabla{\varphi}(x)$; we assume this fact only at this subsection.

Given a class $\mathcal{F}$ of objective functions and an optimality measure $\delta:\mathcal{F}\times \E\to \RR\cup\{+\infty\}$, the \emph{iteration complexity} of a first-order method $\mathcal{M}$ applied to $\varphi \in \mathcal{F}$ for an accuracy $\ep>0$, say $C(\mathcal{M}, \varphi, \delta;\ep)$, is defined as the minimal number of evaluations of a first-order oracle $(\varphi(\cdot),\nabla\varphi(\cdot))$ in the method $\mathcal{M}$ required to find a point $x \in \E$ satisfying $\delta(\varphi,x)\leq \ep$. Then we define the iteration complexity of first-order methods associated with the class $\mathcal{F}$ with respect to the measure $\delta$ by
\begin{equation*}
C(\mathcal{F},\delta;\ep) := \inf_{\mathcal{M}} \sup_{\varphi \in \mathcal{F}}C(\mathcal{M},\varphi,\delta;\ep),
\end{equation*}
where $\mathcal{M}$ runs all first-order methods for the class $\mathcal{F}$ starting from some fixed initial point $x_0 \in \E$.

We are interested in the iteration complexity under the following classes and measures:
\begin{itemize}
\item $\mathcal{F}(x_0,R,L)$ denotes the class of $L$-smooth convex functions $\varphi$ with $X^*\ne \emptyset$ and $\dist(x_0,X^*)\leq R$.
\item Class $\mathcal{F}(x_0,R,L,\kappa,\rho)$: For $R,L,\kappa >0$, $\rho\geq 2$, and $x_0\in \E$, we say that $\varphi$ belongs to the class $\mathcal{F}(x_0,R,L,\kappa,\rho)$ if $\varphi \in \mathcal{F}(x_0,R,L)$ and it satisfies the H\"olderian error bound condition \eqref{HEB}.

Remark that we do not consider the case $\rho\in [1,2[$ because any $L$-smooth convex function cannot admit the H\"olderian error bound condition with exponent $\rho \in [1,2[$.\footnote{
Suppose that $\varphi$ is an $L$-smooth convex function satisfying \eqref{HEB} for some exponent $\rho \in [1,2[$.
For $\rho=1$, Lemma~\ref{lem:HEB} implies $\norm{\nabla{\varphi}(x)} \geq \kappa$ for $x \not\in\Argmin\varphi$. If $\rho \in ]1,2[$, on the other hand,
Lemmas~\ref{lem:gmap} and~\ref{lem:HEB} imply
${\frac{1}{2L}\norm{\nabla{\varphi}(x)}^2 \leq \varphi(x)-\varphi^* \leq \kappa^{-\frac{1}{\rho-1}}\norm{\nabla{\varphi}(x)}^{\frac{\rho}{\rho-1}}}$
for all $x$, which yields $\norm{\nabla{\varphi}(x)} \geq \text{const.}$ for all $x \not\in\Argmin\varphi$. This contradicts to the continuity of $\nabla{\varphi}$ at points in $\Argmin\varphi$.
}

\item We consider the two optimality measures, the functional residual
$
\deltaf(\varphi,x) := \varphi(x) - \inf \varphi$ and the gradient norm $\deltag(\varphi,x) := \norm{\nabla{\varphi}(x)}.
$
\end{itemize}

For the class $\mathcal{F}(x_0,R,L)$, the following lower bound on the iteration complexity holds (by \cite[Section~2.3B]{Nem92} applied to the class of $L$-smooth convex quadratic minimization): 
\begin{equation}\label{low-compl-smooth}
C(\mathcal{F}(x_0,R,L),\deltag;\ep) = \Omega\left(
\min\left\{\dim \E, \sqrt{\frac{LR}{\ep}}\right\}
\right).
\end{equation}
Let us observe a lower bound on the iteration complexity for the class $\mathcal{F}(x_0,R,L,\kappa,\rho)$.

\begin{proposition}\label{prop:low-compl-HEB}
For the class $\mathcal{F}=\mathcal{F}(x_0,R,L,\kappa,\rho)$, we have
$$
C(\mathcal{F},\deltag;\ep) \geq C(\mathcal{F},\deltaf; \ep^*),\quad \text{where}\quad \ep^* := \kappa^{-\frac{1}{\rho-1}}\ep^{\frac{\rho}{\rho-1}}.
$$
\end{proposition}

\begin{proof}
If a first-order method $\mathcal{M}$ applied to $\varphi \in \mathcal{F}$ finds an approximate solution $x \in \E$ satisfying ${\deltag(\varphi,x) \leq \ep}$,
then Lemma~\ref{lem:HEB} implies
$$
\deltaf(\varphi,x) = \varphi(x)-\varphi^*
\leq \kappa^{-\frac{1}{\rho-1}}\norm{\nabla{\varphi}(x)}^{\frac{\rho}{\rho-1}}
= \kappa^{-\frac{1}{\rho-1}}\deltag(\varphi,x)^{\frac{\rho}{\rho-1}}
\leq
\kappa^{-\frac{1}{\rho-1}}\ep^{\frac{\rho}{\rho-1}}=\ep^*.
$$
Therefore, it follows
$
C(\mathcal{M},\varphi,\deltag;\ep) \geq
C(\mathcal{M},\varphi,\deltaf;\ep^*)
$
and we obtain the assertion.
\end{proof}

Under the measure $\deltaf$, the following lower bound is known \cite{NN85}:
\begin{equation}\label{low-compl-obj-HEB}
C(\mathcal{F}(x_0,R,L,\kappa,\rho),\deltaf;\ep) = \left\{
\begin{array}{ll}
\ds
\Omega\left(\min\left\{\dim \E, \sqrt{\frac{L}{\kappa^{\frac{2}{\rho}} \ep^{\frac{\rho-2}{\rho}}}}\right\} \right) &:\rho > 2,\\
\ds
\Omega\left(\min\left\{\dim \E, \sqrt{\frac{L}{\kappa}}\log \frac{\kappa R^2}{\ep}\right\}\right) &:\rho=2.
\end{array}
\right.
\end{equation}
Consequently, by Proposition~\ref{prop:low-compl-HEB}, we obtain lower bounds under the gradient norm measure $\deltag$:
\begin{equation}\label{low-compl-grad-HEB}
C(\mathcal{F}(x_0,R,L,\kappa,\rho),\deltag; \ep) =
\left\{
\begin{array}{ll}
\ds
\Omega\left(\min\left\{\dim \E, \sqrt{\frac{L}{\kappa^{\frac{1}{\rho-1}} \ep^{\frac{\rho-2}{\rho-1}}}}\right\} \right) &:\rho > 2,\\
\ds
\Omega\left(\min\left\{\dim \E, \sqrt{\frac{L}{\kappa}}\log \frac{\kappa R}{\ep}\right\}\right) &:\rho=2.
\end{array}
\right.
\end{equation}


\section{Accelerated Proximal Gradient Method Applied to Regularized Problems}\label{sec:APG}
This section is devoted to review the regularization strategy \cite{LY17,Nes12}, from which we construct adaptive methods in the next section.
We consider to apply an accelerated proximal gradient method to the regularized problem
$$
\min_{x\in \E}\left[\varphi_\sigma(x):=
\varphi(x) + \frac{\sigma}{2}\norm{x-x_0}^2\right],
$$
where $x_0 \in \E$ is a fixed initial point and $\sigma>0$ is a regularization parameter.
Since $\varphi_\sigma$ is strongly convex, it has a unique minimizer. Let
$$x^*_\sigma := \argmin_{x \in \E}\varphi_\sigma(x),\quad \varphi_\sigma^* := \min_{x\in \E} \varphi_\sigma(x).$$
We define the gradient mapping $g_L^\sigma(x)$ for the regularized function $\varphi_\sigma$ in the following manner:
$$
f_\sigma(x):=f(x)+\frac{\sigma}{2}\norm{x-x_0}^2,\qquad
m_L^\sigma(y;x) := f_\sigma(y)+\innprod{\nabla{f}_\sigma(y)}{x-y}+\frac{L}{2}\norm{x-y}^2 + \varPsi(x),
$$
$$
T_L^\sigma(y):=\argmin_{x\in \E}m_L^\sigma(y;x),\qquad
g_L^\sigma(y):=L(y-T_L^\sigma(y)).
$$
The following relations between $\varphi_\sigma(x)$ and $\varphi(x)$ are useful.

\begin{lemma}\label{lem:reg}
\begin{enumerate}[(i)]
\item $\varphi_\sigma(x) \leq \varphi_\sigma(x_0)$ implies $\varphi(x) \leq \varphi(x_0)$.
\item $\norm{x_0-x_\sigma^*} \leq \dist(x_0,X^*)$.
\item We have $\norm{g_L(y)-g_L^\sigma(y)} \leq  \sigma\norm{y-x_0}$ for any $y \in \E$.
\end{enumerate}
\end{lemma}

\begin{proof}
(i) is immediate by $\varphi(x)\leq \varphi_\sigma(x)$ and $\varphi_\sigma(x_0)=\varphi(x_0)$.\\
(ii) For any $x^* \in X^*$, the strong convexity of $\varphi_\sigma$ implies
$
\varphi_\sigma(x^*)\geq \varphi_\sigma^* + \frac{\sigma}{2}\norm{x^*-x_\sigma^*}^2,
$
which can be rewritten as
$\varphi(x^*)-\varphi(x_\sigma^*) \geq \frac{\sigma}{2}(-\norm{x^*-x_0}^2 + \norm{x_\sigma^*-x_0}^2 + \norm{x^*-x_\sigma^*}^2).$
As $\varphi(x^*)-\varphi(x_\sigma^*)\leq 0$, we conclude $\norm{x^*-x_0}^2 \geq \norm{x_\sigma^*-x_0}^2 + \norm{x^*-x_\sigma^*}^2 \geq \norm{x_\sigma^*-x_0}^2$, which proves the assertion.\\
(iii)
In general, if $h$ is a lower-semicontinuous and $\mu$-strongly convex function, we have for any $a,b \in \E$ that\footnote{
By the strong convexity of $\innprod{a}{x}+h(x)$ and $\innprod{b}{x}+h(x)$, we have
$$
\frac{\mu}{2}\norm{x_a^*-x_b^*}^2 \leq [\innprod{a}{x_b^*}+h(x_b^*)] - [\innprod{a}{x_a^*}+h(x_a^*)]
~~\text{and}~~
\frac{\mu}{2}\norm{x_a^*-x_b^*}^2 \leq  [\innprod{b}{x_a^*}+h(x_a^*)]-[\innprod{b}{x_b^*}+h(x_b^*)],
$$
respectively.
Adding them implies
$
\mu\norm{x_a^*-x_b^*}^2 \leq \innprod{a-b}{x_b^*-x_a^*} \leq \norm{a-b}\norm{x_b^*-x_a^*}.
$
}
$$\norm{x_a^*-x_b^*}\leq \frac{\norm{a-b}}{\mu},~~\text{where}~~
x_a^* = \argmin_{x \in \E}\{\innprod{a}{x}+h(x)\},~~ x_b^* = \argmin_{x \in \E}\{\innprod{b}{x}+h(x)\}.$$
This fact implies
$
\norm{T_L(y)-T_L^\sigma(y)} \leq \frac{\norm{\nabla{f}(y)-\nabla{f}_\sigma(y)}}{L} = \frac{\sigma\norm{y-x_0}}{L}.
$
Therefore, we conclude (iii) because of $\norm{g_L(y)-g_L^\sigma(y)}
=  L\norm{T_L(y)-T_L^\sigma(y)}$.
\end{proof}

\subsection{Accelerated Proximal Gradient Method}
We employ Nesterov's accelerated proximal gradient method \cite{Nes13} for solving $\min_{x \in E}\varphi_\sigma(x)$, by regarding the objective function as
\begin{equation}\label{realloc}
\varphi_\sigma(x) = f(x) + \varPsi_\sigma(x),\quad \varPsi_\sigma(x):=\varPsi(x) + \frac{\sigma}{2}\norm{x-x_0}^2.
\end{equation}
The analogy of the definition of $T_L(\cdot)$ for this regularization is 
\begin{equation*}
\tilde{T}_L(y) = \argmin_{x \in \E}\left[f(y) + \innprod{\nabla{f}(y)}{x-y} + \frac{L}{2}\norm{x-y}^2 + \varPsi_\sigma(x)\right].
\end{equation*}
By the identity
$m^\sigma_{L+\sigma}(y;x) = f(y) + \innprod{\nabla{f}(y)}{x-y} +\frac{L}{2}\norm{x-y}^2+ \varPsi_\sigma(x)$,
observe that $\tilde{T}_L(y)=T^\sigma_{L+\sigma}(y)$ holds.
Therefore, the accelerated method $\mathcal{A}(x_0,L_0,\sigma)$ given in \cite[Algorithm~(4.9)]{Nes13} applied to the regularized function \eqref{realloc} can be described as follows:
{\it Let $x_0 \in \E$, $\psi_0(x)=\frac{1}{2}\norm{x-x_0}^2$, $A_0=0$, $L_0 \geq L_{\min}$.
Generate the sequence $\{x_k,\psi_k,M_k,L_k,A_k\}$ by the iteration
$$\{x_{k+1},\psi_{k+1},M_k,L_{k+1},A_{k+1}\} \leftarrow \APGIter_\sigma(x_k,\psi_k,L_k,A_k),$$
for each $k\geq 0$.}
The scheme $\APGIter$ at each iteration is shown in Algorithm~\ref{APG-iter}.
Nesterov's method involves the backtracking line-search procedure to adapt the Lipschitz constant $L_f$ with the estimate $M_k$ (and $L_{k+1}$), which is controlled by the parameters $\gamma_{\rm inc}>1$ and $\gamma_{\rm dec}\geq 1$.

\begin{algorithm}[t]
\caption{
Accelerated Proximal Gradient Iteration
\newline
$\{x_{k+1},\psi_{k+1},M_k,L_{k+1},A_{k+1}\} \leftarrow \APGIter_\sigma(x_k,\psi_k,L_k,A_k)$}
\label{APG-iter}
\textbf{Parameters:} $\gamma_{\rm inc}>1$, $\gamma_{\rm dec}\geq 1$, $L_{\min} \in  ]0,L_f]$.\\
\textbf{Input:} $\sigma>0$, $x_k \in \E$, $\psi_k:\E\to\RR\cup\{+\infty\}$, $L_k>0$, $A_k\geq 0$.
\begin{algorithmic}[1]
\State Compute $v_{k}:=\argmin_{x\in \E} \psi_{k}(x)$\quad (cf. \eqref{APG-iter-v_k-formula}). \label{APG-iter-vk}
\State Set $L:=L_k/\gamma_{\rm inc}$.
\Repeat \label{APG-iter-loop-s}
	\State Set $L:=L\gamma_{\rm inc}$.
	\State Find the largest root $a>0$ of the equation $\frac{a^2}{A_k+a} = 2\frac{1+\sigma A_k}{L}$. \label{a-quad-eq}
	\State Set $y = \frac{A_k x_k +a v_k}{A_k + a}$.
	\State Compute $z=T_{L+\sigma}^\sigma(y)=\argmin_{x}[f(y)+\innprod{\nabla{f}(y)}{x-y}+\frac{L}{2}\norm{x-y}^2 + \varPsi_\sigma(x)]$.
	\State Compute $T_L(z)$ and $T_{L+\sigma}^\sigma(z)$ \label{APG-iter-add-prox}.
	\State Test the conditions
	\begin{subequations}
	\label{APG-iter-test}
	\begin{align}
	&\innprod{\nabla{f}(y)-\nabla{f}(T_{L+\sigma}^\sigma(y))}{y-T_{L+\sigma}^\sigma(y)} \geq \frac{1}{L}\norm{\nabla{f}(y)-\nabla{f}(T_{L+\sigma}^\sigma(y))}^2, \label{APG-iter-test-a}\\
	&\varphi_\sigma(T_{L+\sigma}^\sigma(z)) \leq m_{L+\sigma}^\sigma(z;T_{L+\sigma}^\sigma(z)), \label{APG-iter-test-b}\\
	&\text{[Optional]}~~\varphi(T_L(z)) \leq \varphi(z) \label{APG-iter-test-c}.
	\end{align}
	\end{subequations}
\Until{the conditions in \eqref{APG-iter-test} hold.} \label{APG-iter-loop-e}
\State Define $y_k:=y$, $M_k:=L$, $x_{k+1}:=z=T_{M_k+\sigma}^\sigma(y_k)$, \\
\quad $a_{k+1}:=a$, $A_{k+1}:=A_k+a_{k+1}$, $L_{k+1}:=\max\{L_{\min},M_k/\gamma_{\rm dec}\}$, \label{AGP-iter-upadte}
\State \quad $\psi_{k+1}(x):=\psi_k(x)+a_{k+1}[f(x_{k+1})+\innprod{\nabla{f}(x_{k+1})}{x-x_{k+1}}+\varPsi_\sigma(x)]$.\label{est-seq-formula}
\State Output $\{x_{k+1},\psi_{k+1},M_k,L_{k+1},A_{k+1}\}$.
\end{algorithmic}
\end{algorithm}

\begin{remark}\label{rem:APG-iter-eval}
In one execution of the loop (Lines~\ref{APG-iter-loop-s}--\ref{APG-iter-loop-e}) in $\APGIter$, we have three evaluations of $\varphi(x)$ (at $x\in \{z,T_L(z),T_{L+\sigma}^\sigma(z)\}$), two evaluations of $\nabla{f}(x)$ (at $x \in \{y,z\}$), and three proximal operations $T_{L+\sigma}^\sigma(y),T_{L+\sigma}^\sigma(z),T_L(z)$. There is one proximal operation to compute $v_k$ outside the loop at Line~\ref{APG-iter-vk}. Remark that $v_0=x_0$ holds and,
by the recurrence formula for $\psi_k$ at Line~\ref{est-seq-formula},  $v_k$ for $k\geq 1$ can be computed as
\begin{equation}\label{APG-iter-v_k-formula}
v_k = \prox_{\gamma_k\varPsi}(w_k),\text{ where } \gamma_k = \frac{A_k}{1+\sigma A_k},~~w_k = x_0-\frac{1}{1+\sigma A_k}\sum_{i=1}^ka_i \nabla{f}(x_i).
\end{equation}
\end{remark}

\begin{remark}
Compared with the original Nesterov's method, there are small modifications for our development.
The update $L_{k+1}:=\max\{L_{\min},M_k/\gamma_{\rm dec}\}$ at Line~\ref{AGP-iter-upadte} is slightly different from the original one $L_{k+1}:=M_k/\gamma_{\rm dec}$ in \cite{Nes13}, which affects Lemma~\ref{lem:APG-iter-LS-compl}.
Moreover, we have additional computations of $T_L(z)$ and $T_{L+\sigma}^\sigma(z)$ at Line~\ref{APG-iter-add-prox} in order to test the conditions \eqref{APG-iter-test-b} and \eqref{APG-iter-test-c}.
Remark that, when $L=L_f$ is known, the computation of $T_{L+\sigma}^\sigma(z)$ at Line~\ref{APG-iter-add-prox} can be omitted as it is only used to check the second condition \eqref{APG-iter-test-b}.
We let the third criterion \eqref{APG-iter-test-c} optional; it is independent of the analysis of Algorithm~\ref{alg:AdaAPG}, while we need it in Algorihtm~\ref{alg:rAdaAPG}.
The first condition \eqref{APG-iter-test-a} is equivalent to the one in Nesterov's method (see the proof of Lemma~5 in \cite{Nes13}).
It will be verified in Lemmas~\ref{lem:APG-iter-L} and~\ref{lem:APG-iter-LS-compl} that our modification does not affect the original complexity analysis.
\end{remark}

\begin{lemma}\label{lem:APG-iter-L}
The condition $\eqref{APG-iter-test}$ of $\APGIter$ holds whenever $L\geq L_f$.
\end{lemma}
\begin{proof}
The first condition \eqref{APG-iter-test-a} is satisfied since $f$ is $L$-smooth (e.g., see \cite[Theorem 2.1.5]{Nes04}).
Since $f_\sigma$ is $(L_f+\sigma)$-smooth, the second one \eqref{APG-iter-test-b} holds if $L\geq L_f$ by Lemma~\ref{lem:gmap} (v).
The third one \eqref{APG-iter-test-c} can be verified again by Lemma~\ref{lem:gmap} (v): If $L\geq L_f$, we have
$
\varphi(T_L(z)) \leq m_L(z;T_L(z)) \leq \varphi(z),
$
where the second inequality follows from
\begin{equation}\label{mL-leq-phi}
m_L(z;T_L(z))=\min_{x \in \E}\left\{f(z)+\innprod{\nabla{f}(z)}{x-z}+\varPsi(x)+\frac{L}{2}\norm{x-z}^2\right\} \stackrel{x=z}{\leq}
f(z)+\varPsi(z)=\varphi(z).
\end{equation}
\end{proof}

The following lemma is given in \cite[Lemma 6]{Nes13}, while we rewrite its proof due to the difference in the update of $L_{k+1}$.
\begin{lemma}\label{lem:APG-iter-LS-compl}
\sloppy 
Suppose that $L_k \leq \gamma_{\rm inc}L_f$ holds.
Then, $\APGIter$ ensures the conditions $M_k \leq \gamma_{\rm inc}L_f$ and ${L_{k+1} \leq \max\{L_{\min},\frac{\gamma_{\rm inc}}{\gamma_{\rm dec}}L_f\}\leq \gamma_{\rm inc}L_f}$.
The number of executions of the loop {\rm (Lines~\ref{APG-iter-loop-s}--\ref{APG-iter-loop-e})} is bounded by
$$
1 + \frac{\log \gamma_{\rm dec}}{\log \gamma_{\rm inc}} + \frac{1}{\log \gamma_{\rm inc}} \log \frac{L_{k+1}}{L_k}.
$$
\end{lemma}
\begin{proof}
\sloppy 
Let $n_k \geq 1$ be the number of inner loops so that ${M_k = L_k\gamma_{\rm inc}^{n_k-1}}$.
If $n_k=1$, then we have $M_k = L_k \leq \gamma_{\rm inc}L_f$. If $n_k>1$, then $M_k\leq \gamma_{\rm inc}L_f$ must hold because otherwise ${L_k\gamma_{\rm inc}^{n_k-2}=M_k/\gamma_{\rm inc} > L_f}$ and the $n_k$-th loop cannot occur (by Lemma~\ref{lem:APG-iter-L}). Hence, we conclude $M_k \leq \gamma_{\rm inc}L_f$ and also ${L_{k+1} = \max\{L_{\min},M_k/\gamma_{\rm dec}\} \leq \max\{L_{\min},\frac{\gamma_{\rm inc}}{\gamma_{\rm dec}}L_f\}}$.

Now, the relation $M_k = L_k\gamma_{\rm inc}^{n_k-1}$ implies
$
L_{k+1} = \max\{L_{\min},M_k/\gamma_{\rm dec}\} \geq \frac{1}{\gamma_{\rm dec}}L_k\gamma_{\rm inc}^{n_k-1}.
$
Then, we have $(n_k-1)\log\gamma_{\rm inc} \leq \log \frac{\gamma_{\rm dec}L_{k+1}}{L_k} = \log \gamma_{\rm dec} + \log \frac{L_{k+1}}{L_k}$ and therefore
$
n_k \leq  1 + \frac{\log \gamma_{\rm dec}}{\log \gamma_{\rm inc}} + \frac{1}{\log \gamma_{\rm inc}} \log \frac{L_{k+1}}{L_k} .
$
\end{proof}

The complexity estimate of Nesterov's method is given as follows.

\begin{proposition}\label{prop:APG}
Let $\{x_k,\psi_k,M_k,L_k,A_k\}$ be generated by the accelerated proximal gradient method applied to the regularized objective function \eqref{realloc}, that is,
$$\{x_{k+1},\psi_{k+1},M_k,L_{k+1},A_{k+1}\} \leftarrow \APGIter_\sigma(x_k,\psi_k,L_k,A_k),\quad k=0,1,2,\ldots,$$
with the initialization $x_0 \in \E$, $\psi_0(x)=\frac{1}{2}\norm{x-x_0}^2$, $A_0=0$, $L_0>0$.
\begin{enumerate}[(i)]
\item
If $L_0 \in [L_{\min}, \gamma_{\rm inc}L_f]$, then
we have $L_{\min} \leq L_k \leq M_k \leq \gamma_{\rm inc}L_f$ for all $k \geq 0$.
The total number of executions of the loop {\rm (Lines~\ref{APG-iter-loop-s}--\ref{APG-iter-loop-e})} until the $k$-th iteration is bounded by 
$$
\left[
1+\frac{\log \gamma_{\rm dec}}{\log \gamma_{\rm inc}}\right](k+1) + \frac{1}{\log \gamma_{\rm inc}}\log\frac{L_{k+1}}{L_0}.
$$
\item For each $k\geq 1$, we have
$
A_k \geq \frac{2}{\gamma_{\rm inc}L_f}\left[1+\sqrt{\frac{\sigma}{2\gamma_{\rm inc}L_f}}\right]^{2(k-1)}.
$
\item $\varphi_\sigma(x_k)-\varphi_\sigma^* \leq \frac{\norm{x_0-x_\sigma^*}^2}{2A_k}$ holds for all $k\geq 1$.
\item $\varphi(x_k)\leq \varphi(x_0)$ holds for all $k\geq 1$.
\item
For every $k\geq 1$, we have
$$
\norm{x_k-x_0} \leq \left(1+\frac{1}{\sqrt{\sigma A_k}}\right)\dist(x_0,X^*),
$$
$$
\norm{g_{M_{k-1}}(x_k)} \leq \norm{g_{M_{k-1}+\sigma}^\sigma(x_k)} + \sigma\norm{x_k-x_0}
\leq \left(2\sqrt{\frac{M_{k-1}+\sigma}{A_k}} + \sigma\right)\dist(x_0,X^*).
$$
\end{enumerate}
\end{proposition}
\begin{proof}
(i) 
$L_{\min}\leq L_k\leq M_k$ is clear by the construction. $M_k\leq \gamma_{\rm inc}L_f$
is obtained by applying Lemma~\ref{lem:APG-iter-LS-compl} inductively.

Since $\varPsi_\sigma$ is $\sigma$-strongly convex, 
(ii) follows by \cite[Lemma~8]{Nes13} applied to the objective function $\varphi_\sigma=f+\varPsi_\sigma$.

According to \cite[Lemma 7]{Nes13}, the following relations hold for all $k\geq 0$:
$$
\left\{
\begin{array}{l}
A_k\varphi_\sigma(x_k) \leq \min_{x \in \E}\psi_k(x),\\
\psi_k(x)\leq A_k\varphi_\sigma(x)+\frac{1}{2}\norm{x-x_0}^2,\quad \forall x \in \E.
\end{array}
\right.
$$
Combining them, we obtain
$$
A_k\varphi_\sigma(x_k) \leq \min_{x\in \E}\left\{ A_k\varphi_\sigma(x)+\frac{1}{2}\norm{x-x_0}^2 \right\}, \quad \forall k\geq 0.
$$
Taking $x=x_\sigma^*$ on the right hand side, we obtain (iii).
On the other hand, taking $x=x_0$ yields ${\varphi_\sigma(x_k)\leq \varphi_\sigma(x_0)}$. 
Then, Lemma~\ref{lem:reg} (i) gives the assertion $\varphi(x_k)\leq \varphi(x_0)$.

(v)
By the $\sigma$-strong convexity of $\varphi_\sigma$ and using (iii), we have
$
\frac{\sigma}{2}\norm{x_k-x_\sigma^*}^2 \leq \varphi_\sigma(x_k)-\varphi_\sigma^* \leq \frac{\norm{x_0 -x_\sigma^*}^2}{2A_k},
$
and thus
$
\norm{x_k-x_\sigma^*} \leq \frac{1}{\sqrt{\sigma A_k}}\norm{x_0-x_\sigma^*}.
$
This shows the former inequality of (v):
$$
\norm{x_k-x_0}\leq \norm{x_k-x_\sigma^*} + \norm{x_0-x_\sigma^*} \leq \left(1+\frac{1}{\sqrt{\sigma A_k}}\right) \norm{x_0-x_\sigma^*}
\leq
\left(1+\frac{1}{\sqrt{\sigma A_k}}\right) \dist(x_0,X^*),
$$
where the last inequality follows by Lemma~\ref{lem:reg} (ii).

Since $\varphi_\sigma(T_{M_{k-1}+\sigma}^\sigma(x_k)) \leq m_{M_{k-1}+\sigma}^\sigma(x_k;T_{M_{k-1}+\sigma}^\sigma(x_k))$ holds by \eqref{APG-iter-test-b}, Lemma~\ref{lem:gmap} (iv) yields
$$
\frac{1}{2(M_{k-1}+\sigma)}\norm{g_{M_{k-1}+\sigma}^\sigma(x_k)}^2
\leq \varphi_\sigma(x_k)-\varphi_\sigma^* \leq \frac{\norm{x_0-x_\sigma^*}^2}{2A_k}
\leq \frac{\dist(x_0,X^*)^{2}}{2A_k}.
$$
Thus, $\norm{g_{M_{k-1}+\sigma}^\sigma(x_k)} \leq \sqrt{\frac{M_{k-1}+\sigma}{A_k}}\dist(x_0,X^*)$ holds.
Consequently, using Lemma~\ref{lem:gmap}~(i) and Lemma~\ref{lem:reg}~(iii), we obtain
\begin{align*}
\norm{g_{M_{k-1}}(x_k)} &\leq \norm{g_{M_{k-1}+\sigma}(x_k)}
\leq \norm{g_{M_{k-1}+\sigma}^\sigma(x_k)} + \sigma\norm{x_k - x_0}
\\&\leq  \sqrt{\frac{M_{k-1}+\sigma}{A_k}}\dist(x_0,X^*) + \sigma\left(1+\frac{1}{\sqrt{\sigma A_k}}\right)\dist(x_0,X^*)
\\&\leq \left(2\sqrt{\frac{M_{k-1}+\sigma}{A_k}} + \sigma\right)\dist(x_0,X^*).
\end{align*}
\end{proof}

\subsection{Proximal Gradient Method}

We end this section 
by presenting Algorithm~\ref{alg:PG-iter}, a basic proximal gradient descent with a backtracking strategy to estimate $L_f$ \cite[Algorithm (3.3)]{Nes13}, which will be used in the initialization of our method.
The method consists of evaluations of $\varphi(x)$ at $x \in \{x_k,T_L(x_k)\}$ (which can be omitted if $L_f$ is known), one gradient evaluation  $\nabla{f}(x_k)$, and a proximal operation $T_L(x_k)$, for each guess $L$ of the Lipschitz constant.

\begin{algorithm}
\caption{
Proximal Gradient Iteration
\newline
$\{T_{M_k}(x_k), M_k, L_{k+1}\} \leftarrow \PGIter(x_k,L_k)$
}
\label{alg:PG-iter}
\textbf{Parameters:} $\gamma_{\rm inc}>1$, $\gamma_{\rm dec}\geq 1$, $L_{\min} \in ]0,L_f]$.\\
\textbf{Input:} $x_k \in \E$, $L_k > 0$.
\begin{algorithmic}[1]
\State Set $L:=L_k/\gamma_{\rm inc}$.
\Repeat \label{PG-iter-loop-s}
\State Set $L:=\gamma_{\rm inc}L$.
\State Compute $T_L(x_k)$.
\Until{the following condition holds:} \label{PG-iter-loop-e}
\begin{equation}\label{PG-iter-test}
	\varphi(T_L(x_k))\leq m_L(x_k;T_L(x_k)).
\end{equation}
\State Define $M_k:=L$, $L_{k+1}:=\max\{L_{\min},M_k/\gamma_{\rm dec}\}$.
\State Output $\{T_{M_k}(x_k), M_k, L_{k+1}\}$.
\end{algorithmic}
\end{algorithm}

\begin{lemma}\label{lem:PG-iter}
Let $\{T_{M_k}(x_k), M_k, L_{k+1}\}$ be given by $\PGIter(x_k,L_k)$. Then, the following assertions hold.
\begin{enumerate}[(i)]
\item $\varphi(T_{M_k}(x_k))\leq \varphi(x_k)$ holds.
\item
If $L_k \leq \gamma_{\rm inc}L_f$, then we have $M_k \leq \gamma_{\rm inc}L_f$ and $L_{k+1} \leq \max\{L_{\min},$ $\frac{\gamma_{\rm inc}}{\gamma_{\rm dec}}L_f\} \leq \gamma_{\rm inc}L_f$.
Moreover, the number of executions of the loop {\rm (Lines~\ref{PG-iter-loop-s}--\ref{PG-iter-loop-e})} is bounded by
$$
 1 + \frac{\log \gamma_{\rm dec}}{\log \gamma_{\rm inc}} + \frac{1}{\log \gamma_{\rm inc}} \log \frac{L_{k+1}}{L_k}.
$$
\end{enumerate}
\end{lemma}
\begin{proof}
Since the condition \eqref{PG-iter-test} ensures
$\varphi(T_{M_k}(x_k))\leq m_{M_k}(x_k;T_{M_k}(x_k))\leq \varphi(x_k)$
(recall \eqref{mL-leq-phi} for the second inequality), we conclude (i).
The assertion (ii) can be verified in the same way as Lemma~\ref{lem:APG-iter-LS-compl}.
\end{proof}


\section{Adaptive Proximal Gradient Methods}\label{sec:AdaAPG}

In this section, we propose an adaptive proximal gradient method (Algorithm~\ref{alg:AdaAPG}) and its restart scheme (Algorithm~\ref{alg:rAdaAPG}). We show that these two are nearly optimal for the classes $\mathcal{F}(x_0,R,L)$ and $\mathcal{F}(x_0,R,L,\kappa,\rho)$, respectively.


\subsection{Adaptive Determination of the Regularization Parameter}\label{ssec:AdaAPG}

\begin{algorithm}
\caption{
Adaptive Accelerated Proximal Gradient Method
\newline
$\AdaAPG(x_0,L_{-1},\sigma_0,\ep)$
}
\label{alg:AdaAPG}
\textbf{Parameters:} $\gamma_{\rm inc}>1$, $\gamma_{\rm dec}\geq 1$, $L_{\rm min}\in ]0,L_f]$, $\gamma_{\rm reg}>1$, $\beta\in]0,1]$.\\
\textbf{Input:} $x_0 \in \E$, $L_{-1} \in [L_{\min}, \gamma_{\rm inc}L_f]$, $\sigma_0>0$, $\ep>0$.
\begin{algorithmic}[1]
\State $L_{0}:= L_{-1}$.
\For{$j=0,1,2,\ldots$}
	\State $\sigma_j := \sigma_0/\gamma_{\rm reg}^j$.
	\State $\psi_{0}(x):=\frac{1}{2}\norm{x-x_{0}}^2$, $A_{0}:=0$.
	\Repeat{ \textbf{for} $k=0,1,2,\ldots$}
		\State $\{x_{k+1},\psi_{k+1},M_{k},L_{k+1},A_{k+1}\} \leftarrow \APGIter_{\sigma_j}(x_{k},\psi_{k},L_{k},A_{k})$.
		\If{$\norm{g_{M_{k}}(x_{k+1})} \leq \ep$}
			\State output $\{\sigma_j, x_{k+1},T_{M_{k}}(x_{k+1}),M_{k},L_{k+1}\}$ and terminate the algorithm.
		\EndIf
	\Until{$A_{k+1} \geq \frac{2(M_{k}+\sigma_j)}{\beta^2\sigma_j^2}$}.\label{A-growth}
\State $L_{0} := L_{k+1}$.
\EndFor
\end{algorithmic}
\end{algorithm}

For solving \eqref{main-prob} under the measure $\norm{g_L(x)}$, we propose the adaptive accelerated proximal gradient method $\AdaAPG$ shown in Algorithm~\ref{alg:AdaAPG}, which can be seen as a simple extension of the regularization technique \cite{Nes12} introducing a guess-and-check procedure to adapt the regularization parameter $\sigma$.
The $j$-th outer loop of $\AdaAPG$ corresponds to applying Nesterov's accelerated proximal gradient method to the regularized problem $\min_{x\in \E} \varphi_{\sigma_j}(x)$, where $\sigma_j=\sigma_0/\gamma_{\rm reg}^j$~($\gamma_{\rm reg}>1$). We stop Nesterov's method when it successfully finds an $\ep$-solution or when it iterates excessively as detected by the growth condition of $A_{k+1}$ at Line~\ref{A-growth}. The growth of $A_{k+1}$ is used as the criterion that the current guess of $\sigma_j$ is not desirable, and then we restart Nesterov's method reallocating the regularization parameter as $\sigma_{j+1}:=\sigma_j/\gamma_{\rm reg}$. The proposed method involves the parameter $\beta\in]0,1]$, which controls the timing when $\sigma_j$ should be decreased;
smaller $\beta$ imposes larger $A_{k+1}$ in the condition at Line~\ref{A-growth}, and then we examine the current guess $\sigma_j$ for a longer time.

The next lemma shows what happens if an outer loop is completed without the termination of the algorithm.

\begin{lemma}\label{lem:AdaAPG-loop}
In $\AdaAPG$,
suppose that a $j$-th outer loop is completed satisfying the criterion $A_{k+1}\geq \frac{2(M_{k}+\sigma_j)}{\beta^2\sigma_j^2}$ for some $k\geq 0$. Then, we have the inequalities
$$
\norm{x_{k+1}-x_0} \leq \left(1+\frac{\beta}{\sqrt{2}}\right)\dist(x_0,X^*),
\qquad
\norm{g_{M_{k}}(x_{k+1})} \leq (1+\sqrt{2}\beta)\sigma_j\dist(x_0,X^*).
$$
Moreover, the number of inner iterations is bounded as follows.
$$
k+1 \leq 2+ \left(\sqrt{\frac{2\gamma_{\rm inc}L_f}{\sigma_j}}+1\right)\log\frac{\gamma_{\rm inc}L_f+\sigma_j}{\beta\sigma_j}.
$$
\end{lemma}
\begin{proof}
The bounds on $\norm{x_{k+1}-x_0}$ and $\norm{g_{M_{k}}(x_{k+1})}$ can be obtained by Proposition~\ref{prop:APG} (v), applying the assumption $A_{k+1}\geq \frac{2(M_{k}+\sigma_j)}{\beta^2\sigma_j^2}$.
To show the bound on $k+1$, suppose $k\geq 1$ (the result is clear when $k=0$).
By the assumption on $k$ and Proposition~\ref{prop:APG}~(i), we have
$$A_{k} < \frac{2(M_{k-1}+\sigma_j)}{\beta^2\sigma_j^2} \leq \frac{2(\gamma_{\rm inc}L_f+\sigma_j)}{\beta^2\sigma_j^2}.$$
On the other hand, Proposition~\ref{prop:APG} (ii) implies
$$
A_{k} \geq  \frac{2}{\gamma_{\rm inc}L_f}\left[1+\sqrt{\frac{\sigma_j}{2\gamma_{\rm inc}L_f}}\right]^{2(k-1)} \geq
 \frac{2}{\gamma_{\rm inc}L_f}\exp\left(2(k-1) \frac{1}{\sqrt{\frac{2\gamma_{\rm inc}L_f}{\sigma_j}}+1}\right),
$$
where the second inequality is due to the fact\footnote{
In fact, since the derivative $\log(1+x)$ of the function $h(x):=(1+x)\log(1+x)-x$ is increasing and vanishes at $x=0$, we have $\min_{x>-1}h(x)=h(0)=0$.
}
$1+x \geq \exp\frac{x}{1+x}~(\forall x > -1)$.
Therefore,
\begin{align*}
\frac{2}{\gamma_{\rm inc}L_f}\exp\left(2(k-1) \frac{1}{\sqrt{\frac{2\gamma_{\rm inc}L_f}{\sigma_j}}+1}\right) \leq \frac{2(\gamma_{\rm inc}L_f+\sigma_j)}{\beta^2\sigma_j^2}\\
\Longrightarrow
k+1 \leq 2+ \frac{1}{2}\left(\sqrt{\frac{2\gamma_{\rm inc}L_f}{\sigma_j}}+1\right)\log\frac{\gamma_{\rm inc}L_f(\gamma_{\rm inc}L_f+\sigma_j)}{\beta^2\sigma_j^2}.
\end{align*}
The assertion follows by relaxing $\gamma_{\rm inc}L_f \leq \gamma_{\rm inc}L_f + \sigma_j$.
\end{proof}

In view of Proposition~\ref{prop:APG} (i) and Remark~\ref{rem:APG-iter-eval}, the total number of executions of $\APGIter$, say $N$, determines the iteration complexity of $\AdaAPG$. For instance, the total number of evaluations of $\nabla{f}(\cdot)$ in $\AdaAPG$ is bounded by
$$
2\left[
1+\frac{\log \gamma_{\rm dec}}{\log \gamma_{\rm inc}}\right]N + \frac{2}{\log \gamma_{\rm inc}}\log\frac{\gamma_{\rm inc}L_f}{L_{-1}}.
$$

\begin{theorem}\label{thm:AdaAPG}
In $\AdaAPG$, let $N$ be the total number of executions of $\APGIter$.
Denote
\begin{equation}\label{sigma-th}
\sigma(x_0,\ep):=\frac{\ep}{(1+\sqrt{2}\beta)\dist(x_0,X^*)},
\end{equation}
where we let $\sigma(x_0,\ep):=+\infty$ when $x_0 \in X^*$.
Then, the following assertions hold.
\begin{enumerate}[(i)]
\item
The algorithm terminates at the $j$-th outer loop whenever
$\sigma_j \leq \sigma(x_0,\ep)$.
Moreover, we have ${\sigma_j \geq \sigma(x_0,\ep)/\gamma_{\rm reg}}$ for all $j>0$.

\item
Suppose that the algorithm terminates at $j=\ell$ for some $\ell \geq 0$.
Then, $N$ is at most
\begin{align*}
&\frac{\sqrt{2\gamma_{\rm inc}L_f}}{\sqrt{\gamma_{\rm reg}}-1}\left[\sqrt{\gamma_{\rm reg}}\sqrt{\frac{1}{\sigma_\ell}} - \sqrt{\frac{1}{\sigma_0}}\right]\log\frac{\gamma_{\rm inc}L_f+\sigma_\ell}{\beta\sigma_\ell}+ \left(1+\log_{\gamma_{\rm reg}}\frac{\sigma_0}{\sigma_\ell}\right)\left(2+\log\frac{\gamma_{\rm inc}L_f+\sigma_\ell}{\beta\sigma_\ell}\right).
\end{align*}
\item
If $\sigma_0\leq \sigma(x_0,\ep)$, then
$
N \leq 2+ \left(\sqrt{\frac{2\gamma_{\rm inc}L_f}{\sigma_0}}+1\right)\log\frac{\gamma_{\rm inc}L_f+\sigma_0}{\beta\sigma_0}.
$

\item
If $\sigma_0 \geq \sigma(x_0,\ep)$, then $N$ is at most
\begin{align}
&\frac{\sqrt{2\gamma_{\rm inc}L_f}}{\sqrt{\gamma_{\rm reg}}-1}\left[\gamma_{\rm reg}\sqrt{\frac{1}{\sigma(x_0,\ep)}} - \sqrt{\frac{1}{\sigma_0}} \right]\log\left(\frac{\gamma_{\rm reg}\gamma_{\rm inc}L_f}{\beta\sigma(x_0,\ep)}+\frac{1}{\beta}\right) \nonumber\\
&+ \left(2+\log_{\gamma_{\rm reg}}\frac{ \sigma_0}{\sigma(x_0,\ep)}\right)\left[2+\log\left(\frac{\gamma_{\rm reg}\gamma_{\rm inc}L_f}{\beta\sigma(x_0,\ep)}+\frac{1}{\beta}\right)\right] \nonumber\\
=~&
O\left(\sqrt{\frac{L_f\dist(x_0,X^*)}{\ep}}\log \frac{L_f\dist(x_0,X^*)}{\ep} + \log \frac{\sigma_0 \dist(x_0,X^*)}{\ep} \log \frac{L_f\dist(x_0,X^*)}{\ep}\right).
\label{AdaAPG-compl-O}
\end{align}
\end{enumerate}
\end{theorem}

\begin{proof}
(i)
\sloppy 
By Lemma~\ref{lem:AdaAPG-loop}, the algorithm must terminate at the $j$-th loop whenever we have
$
{(1+\sqrt{2}\beta)\sigma_j\dist(x_0,X^*) \leq \ep}
$,
or equivalently,
$
\sigma_j \leq \frac{\ep}{(1+\sqrt{2}\beta)\dist(x_0,X^*)}=\sigma(x_0,\ep).
$
Thus, we obtain the former assertion of (i).
To see the latter assertion,
assume that $\sigma_j < \sigma(x_0,\ep)/\gamma_{\rm reg}$ holds for some $j > 0$. Then, the previous $(j-1)$-th loop satisfies $\sigma_{j-1}< \sigma(x_0,\ep)$ so that the $j$-th loop is not executed by the former assertion.
This verifies the latter assertion of (i).

(ii) Since $\sigma_\ell=\sigma_0/\gamma_{\rm reg}^\ell$, we have $\ell = \log_{\gamma_{\rm reg}}\sigma_0/\sigma_\ell$.
Using Lemma~\ref{lem:AdaAPG-loop}, $N$ is bounded as follows.
\begin{align*}
N
&\leq
\sum_{j=0}^\ell \left[
2+ \left(\sqrt{\frac{2\gamma_{\rm inc}L_f}{\sigma_j}}+1\right)\log\frac{\gamma_{\rm inc}L_f+\sigma_j}{\beta\sigma_j}
\right]\\
&\leq 2(\ell+1) + \sqrt{2\gamma_{\rm inc}L_f}\log\frac{\gamma_{\rm inc}L_f+\sigma_\ell}{\beta\sigma_\ell}\sum_{j=0}^\ell \sqrt{\frac{1}{\sigma_j}}
+ (\ell+1)\log\frac{\gamma_{\rm inc}L_f+\sigma_\ell}{\beta\sigma_\ell},
\end{align*}
where the second inequality is due to $\sigma_j \geq \sigma_\ell~(\forall j \leq \ell)$.
Note that
\begin{align*}
\sum_{j=0}^\ell \sqrt{\frac{1}{\sigma_j}}
&=
\sum_{j=0}^\ell \sqrt{\frac{1}{\sigma_0/\gamma_{\rm reg}^j}}
=
\sqrt{\frac{1}{\sigma_0}} \sum_{j=0}^\ell \sqrt{\gamma_{\rm reg}}^{j}
=
\sqrt{\frac{1}{\sigma_0}}\frac{\sqrt{\gamma_{\rm reg}}^{\ell+1} -1}{\sqrt{\gamma_{\rm reg}} -1}\\
&=
\sqrt{\frac{1}{\sigma_0}}\frac{\sqrt{\gamma_{\rm reg}}\sqrt{\sigma_0/\sigma_\ell} -1}{\sqrt{\gamma_{\rm reg}} -1}
=
\frac{1}{\sqrt{\gamma_{\rm reg}}-1}\left[\sqrt{\gamma_{\rm reg}}\sqrt{\frac{1}{\sigma_\ell}} - \sqrt{\frac{1}{\sigma_0}}\right].
\end{align*}
Therefore, we see that
$$
N\leq (\ell+1)\left(2+\log\frac{\gamma_{\rm inc}L_f+\sigma_\ell}{\beta\sigma_\ell}\right)
+
\frac{\sqrt{2\gamma_{\rm inc}L_f}}{\sqrt{\gamma_{\rm reg}}-1}\left[\sqrt{\gamma_{\rm reg}}\sqrt{\frac{1}{\sigma_\ell}} - \sqrt{\frac{1}{\sigma_0}}\right]\log\frac{\gamma_{\rm inc}L_f+\sigma_\ell}{\beta\sigma_\ell}.
$$
The assertion follows by substituting $\ell = \log_{\gamma_{\rm reg}}{\sigma_0}/{\sigma_\ell}$.

In view of (i), the assertions (iii) and (iv) follow by applying (ii) with $\sigma_\ell = \sigma_0$ and $\sigma_\ell \geq \sigma(x_0,\ep)/\gamma_{\rm reg}$, respectively.
The big-$O$ expression \eqref{AdaAPG-compl-O} is obtained by substituting the definition of $\sigma(x_0,\ep)$.
\end{proof}

If $\sigma_0$ is chosen appropriately, then 
the complexity estimates in Theorem~\ref{thm:AdaAPG} (iii) and (vi) match
the lower complexity bound \eqref{low-compl-smooth} for the class $\mathcal{F}(x_0,R,L)$, up to a logarithmic factor.
Nesterov's regularization technique \cite{Nes12} essentially corresponds to the ideal choice $\sigma_0=\sigma(x_0,\ep)$, which requires $\dist(x_0,X^*)$ to be known.
Here, we show a simple example to choose $\sigma_0$ so that $\AdaAPG$ attains the near optimality without knowing $\dist(x_0,X^*)$.

\begin{corollary}\label{cor:AdaAPG}
Given $x_{0} \in \E$ and $L_{-1} \in [L_{\min},\gamma_{\rm inc}L_f]$, apply
$\{T_M(x_0),M,L\}\leftarrow \PGIter(x_{0},L_{-1})$.
For any $\ep>0$ such that $\norm{g_M(x_0)} \geq \ep$, 
choose $\sigma_0$ from the interval
\begin{equation}\label{sigma0-choice}
\sigma_0 \in \left[\frac{2\ep M}{(1+\sqrt{2}\beta)\norm{g_{M}(x_0)}}, \frac{2M}{1+\sqrt{2}\beta}\right].
\end{equation}
Then, we have
\begin{equation}\label{simga0-bound}
\sigma(x_0,\ep) \leq \sigma_0 \leq \frac{2\gamma_{\rm inc}L_f}{1+\sqrt{2}\beta}.
\end{equation}
Consequently, $\AdaAPG(x_0,L_{-1},\sigma_0,\ep)$ finds $x_k \in \E$ and $M_k>0$ satisfying $\norm{g_{M_k}(x_k)}\leq \ep$ with the iteration complexity at most
\begin{equation}\label{compl-bound-near-opt}
O\left(\sqrt{\frac{L_f\dist(x_0,X^*)}{\ep}}\log \frac{L_f\dist(x_0,X^*)}{\ep}\right).
\end{equation}
\end{corollary}
\begin{proof}
Since $\varphi(T_M(x_0)) \leq m_M(x_0;T_M(x_0))$ holds by the condition \eqref{PG-iter-test} in $\PGIter$, Lemma~\ref{lem:gmap} (iv) implies
$$\dist(x_0,X^*) \geq \frac{1}{2M}\norm{g_{M}(x_0)} \geq \frac{1}{2M}\ep \geq \frac{1}{2\gamma_{\rm inc}L_f}\ep,$$
where the last inequality follows from Lemma~\ref{lem:PG-iter} (ii).
This can be rewritten as
$$
\sigma(x_0,\ep) \equiv \frac{\ep}{(1+\sqrt{2}\beta)\dist(x_0,X^*)} \leq \frac{2\ep M}{(1+\sqrt{2}\beta)\norm{g_{M}(x_0)}} \leq \frac{2M}{1+\sqrt{2}\beta} \leq \frac{2\gamma_{\rm inc}L_f}{1+\sqrt{2}\beta}.
$$
Therefore, \eqref{sigma0-choice} implies \eqref{simga0-bound}.
In particular, we can apply Theorem~\ref{thm:AdaAPG}~(iv) to obtain the complexity bound \eqref{AdaAPG-compl-O}, whose second term is dominated by the first one due to the upper bound \eqref{simga0-bound} of $\sigma_0$.
\end{proof}

Finally, we make an observation on the convergence of the proposed method.
Let us consider the execution of $\AdaAPG$ with $\ep=0$, which is possibly an infinite step algorithm (Notice that $\ep$ appears only in the stopping criterion).
Then, the choice
$\sigma_0:=2M/(1+2\sqrt{\beta})$
given in Corollary~\ref{cor:AdaAPG}
ensures the nearly optimal complexity bound \eqref{compl-bound-near-opt} for \emph{every} $\ep\in ]0,\|{g_{M}(x_0)}\|[$.
This means that, with this choice of $\sigma_0$, the algorithm $\AdaAPG$ with $\ep=0$ yields a \emph{nearly optimal convergence} with respect to $\norm{g_L(x)}$.


\newcommand{\outerXinit}[1]{{x^{(#1)}}}
\newcommand{\outerXplus}[1]{{x^{(#1)}_+}}
\newcommand{\xinit}{{x^{(0)}}}
\newcommand{\Linit}{{L^{(-1)}}}
\newcommand{\sigmainit}{{\sigma^{(0)}}}
\newcommand{\outerM}[1]{{M^{(#1)}}}
\newcommand{\outerL}[1]{{L^{(#1)}}}
\newcommand{\outersigma}[1]{{\sigma^{(#1)}}}
\newcommand{\outerep}[1]{{\ep^{(#1)}}}
\newcommand{\innerM}[2]{{M_{#2}}}
\newcommand{\innerL}[2]{{L_{#2}}}
\newcommand{\innerx}[2]{{x_{#2}}}
\newcommand{\innersigma}[2]{{\sigma_{#2}}}
\newcommand{\innerphi}[2]{{\varphi_{#2}}}
\newcommand{\innerA}[2]{{A_{#2}}}
\newcommand{\innery}[2]{{y_{#2}}}

\subsection{Adaptive Restart Algorithm under the H\"olderian Error Bound Condition}\label{ssec:rAdaAPG}

In this section, given an initial point $\xinit \in \E$, we assume that the objective function $\varphi$ admits the H\"olderian error bound condition
\begin{equation}\label{HEB-review}
\varphi(x)-\varphi^* \geq \kappa\dist(x,X^*)^\rho,\quad \forall x\in \lev_{\varphi}(\varphi(\xinit)),
\end{equation}
for some $\kappa>0$ and $\rho\geq 1$.
For this case, we propose the restart scheme  $\rAdaAPG$ described in Algorithm~\ref{alg:rAdaAPG}.
Namely, given a current solution $\outerXinit{t}$, we apply a proximal gradient iteration $\outerXplus{t}:=T_{\outerM{t}}(\outerXinit{t})$, from which we start $\AdaAPG$ to find the next $\outerXinit{t+1}$ reducing the gradient mapping norm at the ratio $\theta \in ]0,1[$:
$$
\norm{g_{\outerM{t+1}}(\outerXinit{t+1})} \leq \theta\norm{g_{\outerM{t}}(\outerXinit{t})}.
$$
Remarkably, the regularization parameter $\outersigma{t}$ is input to $\AdaAPG$, which adaptively finds the next $\outersigma{t+1}$.
Therefore, the next $\outerXinit{t+1}$ can be seen as an approximate solution to the regularized problem
$$
\min_{x \in \E}\left[\varphi(x)+\frac{\outersigma{t+1}}{2}\norm{x-\outerXplus{t}}^2\right],
$$
computed by an accelerated proximal gradient method.
Finally, we note that the initial regularization parameter $\sigmainit$ may be determined depending on the result of Line~\ref{rAdaAPG-init}, as observed in Corollary~\ref{cor:rAdaAPG}.

\begin{algorithm}
\caption{
Restart Scheme for Adaptive Accelerated Proximal Gradient Method
\newline
$\rAdaAPG (\xinit,\Linit,\sigmainit,\ep)$
}
\label{alg:rAdaAPG}
\textbf{Parameters:}\\
\quad $\gamma_{\rm inc}>1$, $\gamma_{\rm dec}\geq 1$, $L_{\rm min}\in ]0,L_f]$ (for backtracking line-search);\\
\quad $\gamma_{\rm reg}>1$ (the ratio to decrease the regularization parameter);\\
\quad $\beta\in ]0,1]$  (controls the timing when the current regularization parameter must be decreased);\\
\quad $\theta \in ]0,1[$ (the ratio reducing the residue per iteration).\\
\textbf{Input:} $\xinit \in \E$, $\Linit \in [L_{\min},\gamma_{\rm inc}L_f]$, $\sigmainit>0$, $\ep>0$.
\begin{algorithmic}[1]
\State Compute $\{\outerXplus{0},\outerM{0},\outerL{0}\} \leftarrow \PGIter(\xinit,\Linit)$ \label{rAdaAPG-init}.
\For{$t=0,1,2,\ldots$}
	\State \textbf{if} $\norm{g_{\outerM{t}}(\outerXinit{t})} \leq \ep$ \textbf{then} output $\{\outerXinit{t},\outerM{t},\outerXplus{t}\}$ and terminate the algorithm. \label{rAdaAPG-terminate}
	\State $\outerep{t}:= \theta\norm{g_{\outerM{t}}(\outerXinit{t})}$.
	\State $\{\outersigma{t+1},\outerXinit{t+1},\outerXplus{t+1},\outerM{t+1},\outerL{t+1}\} \leftarrow \AdaAPG(\outerXplus{t},\outerL{t},\outersigma{t},\ep^{(t)})$
	\State (with testing the optional condition \eqref{APG-iter-test-c} of $\APGIter$).
\EndFor
\end{algorithmic}
\end{algorithm}

\subsubsection{Iteration Complexity Results and Near Optimality}
To show the iteration complexity result, observe that
the total number $N$ of executions of $\APGIter$ determines the complexity of $\rAdaAPG$. For instance, the total number of evaluations of $\nabla{f}(\cdot)$ in $\rAdaAPG$ can be bounded by (recall Proposition~\ref{prop:APG} (i), Lemma~\ref{lem:PG-iter} (ii), and Remark~\ref{rem:APG-iter-eval})
$$
\left[
1+\frac{\log \gamma_{\rm dec}}{\log \gamma_{\rm inc}}\right](2N+1) + \frac{2}{\log \gamma_{\rm inc}}\log\frac{\gamma_{\rm inc}L_f}{\Linit}.
$$
We focus on the bound on $N$ in the remaining of this section.

To describe the iteration complexity bounds, we define $N(\ep,\sigma_*,C)$ and $\bar\sigma$ as follows.
Given $\ep,\sigma_*,C>0$, let $N(\ep,\sigma_*,C):=0$ if $\norm{g_{\outerM{0}}(\xinit)} \leq \ep$ and otherwise
\begin{align}
N(\ep,\sigma_*,C) &
:= \left(1+ \log_{1/\theta}\frac{\norm{g_{\outerM{0}}(\xinit)}}{\ep} + \log_{\gamma_{\rm reg}}\frac{\sigmainit}{\sigma_*}\right)\left(2+\log\frac{\gamma_{\rm inc}L_f+\sigma_*}{\beta\sigma_*}\right) \nonumber\\
&\quad +
\frac{\sqrt{2\gamma_{\rm inc}L_f}}{\sqrt{\gamma_{\rm reg}}-1}\left[
\sqrt{\frac{1}{\sigma_*}} - \sqrt{\frac{1}{\sigmainit}}\right]\log\frac{\gamma_{\rm inc}L_f+\sigma_*}{\beta\sigma_*} 
+ C\sqrt{2\gamma_{\rm inc}L_f}\log\frac{\gamma_{\rm inc}L_f+\sigma_*}{\beta\sigma_*}. \label{N-def}
\end{align}
Moreover, we define
\begin{equation}\label{bar-sigma-def}
\bar\sigma = \left\{
\begin{array}{ll}
\ds
\frac{\theta}{1+\sqrt{2}\beta}
\kappa^{\frac{1}{\rho-1}}\left(\frac{L_f}{L_{\min}}+1\right)^{-\frac{1}{\rho-1}}
\ep^{\frac{\rho-2}{\rho-1}}
& (\text{if } \rho \geq 2),\\[3truemm]
\ds
\frac{\theta}{1+\sqrt{2}\beta}
\kappa^{\frac{2}{\rho}}\left(\frac{L_f}{L_{\min}}+1\right)^{-1}(\varphi(\xinit)-\varphi^*)^{-\frac{2-\rho}{\rho}}
& (\text{if } 1\leq \rho<2).
\end{array}
\right.
\end{equation}
The next theorem provides the descriptions of iteration complexity bounds, which have complicated expressions to explicate their dependence on parameters. A simplified form is presented in Corollary~\ref{cor:rAdaAPG}.
Their proofs are given in Section~\ref{sec:proof-main-result}.
\begin{theorem}\label{thm:rAdaAPG}
Assume that the H\"olderian error bound condition \eqref{HEB-review} holds.
In $\rAdaAPG$, 
let $N$ be the total number of executions of $\APGIter$.
Denote
\begin{equation}\label{simga*-def}
\sigma_* := \left\{
\begin{array}{ll}
\sigmainit & (\text{if } \sigmainit \leq \bar\sigma),\\
\bar\sigma/\gamma_{\rm reg} & (\text{otherwise}),
\end{array}
\right.
\end{equation}
for  $\bar\sigma$ defined in \eqref{bar-sigma-def}.
Also, let  $N(\cdot,\cdot,\cdot)$ be defined by    \eqref{N-def}.
Then, the following assertions hold.
\begin{enumerate}[(i)]
\item If $\rho=2$, then $N\leq N(\ep,\sigma_*,C)$ holds with
$
C =
\sqrt{\frac{1}{\sigma_*}}\left(1+\log_{1/\theta}\frac{\norm{g_{\outerM{0}}(\xinit)}}{\ep}\right).
$

\item Suppose $\rho>2$.
If $\sigmainit\geq\bar\sigma$, then  $N\leq N(\ep,\sigma_*,C)$ holds with
\begin{align*}
C=&
\sqrt{\frac{1}{\sigmainit}}\left(1+\frac{\rho-1}{\rho-2}\log_{1/\theta}\frac{\sigma_*^{(0)}}{\min(\sigma_*^{(0)},\sigmainit)}\right)+
\frac{\sqrt{\gamma_{\rm reg}}}{1-\sqrt{\theta}^{\frac{\rho-2}{\rho-1}}}
\left(
\sqrt{\frac{1}{\bar\sigma}}
-
\sqrt{\theta}^{\frac{\rho-2}{\rho-1}}
\sqrt{\frac{1}{\sigmainit}}
\right),
\end{align*}
where
$
\sigma_*^{(0)}:=\frac{\theta}{1+\sqrt{2}\beta}\cdot
\kappa^{\frac{1}{\rho-1}}\left(\frac{L_f}{L_{\min}}+1\right)^{-\frac{1}{\rho-1}}
\|g_{\outerM{0}}(\xinit)\|^{\frac{\rho-2}{\rho-1}}.
$
Otherwise, it follows $\outersigma{t}=\sigmainit$ for all $t\geq 0$ and $N\leq N(\ep,\sigmainit,C)$ holds with
$
C= 
\sqrt{\frac{1}{\sigmainit}}\left(1+\log_{1/\theta}\frac{\norm{g_{\outerM{0}}(\xinit)}}{\ep}\right).
$
\item Suppose $\rho\in ]1,2[$.
Then, we have
\begin{equation}\label{N-bd-superlinear}
N \leq 1+ N(\max(\ep,\ep_*),\sigma_*,C) + \left(\log\frac{1}{\rho-1}\right)^{-1}\left(\log\log \frac{\ep_*}{\theta^{\frac{\rho-1}{2-\rho}}\min(\ep,\ep_*)}-\log\log\frac{1}{\theta^{\frac{\rho-1}{2-\rho}}}\right),
\end{equation}
where
$
C =
\sqrt{\frac{1}{\sigma_*}}\left(1+\log_{1/\theta}\frac{\norm{g_{\outerM{0}}(\xinit)}}{\max(\ep,\ep_*)}  \right)
$ and
\begin{equation}\label{rho-12-ep*-def}
\ep_*=
\left[(1+\sqrt{2})\theta^{-1}(\gamma_{\rm inc}L_f+\sigmainit)\right]^{-\frac{\rho-1}{2-\rho}}
{\kappa^{\frac{1}{2-\rho}}}\left(\frac{L_f}{L_{\min}}+1\right)^{-\frac{1}{2-\rho}}.
\end{equation}
\item
\sloppy 
If $\rho=1$, then we have
$
N \leq 1 + N(\max(\ep,\ep_*),\sigma_*,C)
$,
where
$
\ep_* = \kappa\left(\frac{L_f}{L_{\min}} + 1\right)^{-1}
$ and
$
{C=
\sqrt{\frac{1}{\sigma_*}}\left(1+\log_{1/\theta}\frac{\norm{g_{\outerM{0}}(\xinit)}}{\max(\ep,\ep_*)}\right)}
$.
Moreover, if $\ep<\ep_*$, then $\outerXplus{t}$ in the output of the algorithm must be an optimal solution to the original problem \eqref{main-prob}.
\end{enumerate}
\end{theorem}

\begin{corollary}\label{cor:rAdaAPG}
Let $\xinit \in \E$ and $\Linit\in [L_{\min},\gamma_{\rm inc}L_f]$. 
Assume that the H\"olderian error bound condition \eqref{HEB-review} holds.
After the initialization step {\rm (Line~\ref{rAdaAPG-init})} in $\rAdaAPG$, determine $\sigmainit$ by
\begin{equation}\label{sigma0-choice-HEB}
\sigmainit:=\frac{2\ep^{(0)} M}{(1+\sqrt{2}\beta)\norm{g_M(\outerXplus{0})}},
~~\text{where}~~
\{T_{M}(\outerXplus{0}),M,L\}\leftarrow \PGIter(\outerXplus{0},\outerL{0}).
\end{equation}
Then, $\rAdaAPG(\xinit,\Linit,\sigmainit,\ep)$ finds $\outerXinit{t} \in \E$ and $\outerM{t}>0$ such that $\norm{g_{\outerM{t}}(\outerXinit{t})}\leq \ep$ with iteration complexity at most the following quantities, where we denote
$$g_0:=g_{\outerM{0}}(\xinit),\quad \Delta_0 := \varphi(\xinit)-\varphi^*,$$
$\ep_*$ is defined by \eqref{rho-12-ep*-def},
and we regard $\theta,\gamma_{\rm inc},\gamma_{\rm reg},\beta,\frac{L_f}{L_{\min}}$ as constants\footnote{This assumption on $L_f/L_{\min}$ means that we know a good estimate $L_{\min}$ close to $L_f$ in advance.}.
\begin{equation}\label{rAdaAPG-compl-O}
\left\{
\begin{array}{ll}
O\left(\left[\log\frac{\norm{g_0}}{\ep} + \sqrt{\frac{L_f}{\kappa^{\frac{1}{\rho-1}}\ep^{\frac{\rho-2}{\rho-1}}}}\right]\log\frac{L_f}{\ep^{\frac{\rho-2}{\rho-1}}} \right) & (\rho>2),\\
O\left(\sqrt{\frac{L_f}{\kappa}}\log\frac{L_f}{\kappa}\log \frac{\norm{g_0}}{\ep}\right) & (\rho=2),\\
O\left(
\sqrt{\frac{L_f\Delta_0^{\frac{2-\rho}{\rho}}}{\kappa^{\frac{2}{\rho}}}}
\log {\frac{L_f\Delta_0^{\frac{2-\rho}{\rho}}}{\kappa^{\frac{2}{\rho}}}}
\log \frac{\norm{g_0}L_f^{\frac{\rho-1}{2-\rho}}}{\kappa^{\frac{1}{2-\rho}}}
+ (\log\frac{1}{\rho-1})^{-1}\log\log\frac{\kappa^{\frac{1}{2-\rho}}}{L_f^{\frac{\rho-1}{2-\rho}} \ep}
\right) & (\rho \in ]1,2[, \ep\leq \ep_*),\\
O\left(
\sqrt{\frac{L_f\Delta_0^{\frac{2-\rho}{\rho}}}{\kappa^{\frac{2}{\rho}}}}
\log {\frac{L_f\Delta_0^{\frac{2-\rho}{\rho}}}{\kappa^{\frac{2}{\rho}}}}
\log \frac{\norm{g_0}}{\ep}
\right) & (\rho \in ]1,2[, \ep\geq \ep_*),\\
O\left(\sqrt{\frac{L_f\Delta_0}{\kappa^2}}\log\frac{L_f\Delta_0}{\kappa^2}\log\frac{\norm{g_0}}{\max(\kappa,\ep)}\right) & (\rho=1).
\end{array}
\right.
\end{equation}
\end{corollary}

Let us observe the consequences of Corollary~\ref{cor:rAdaAPG}.
Notice that
the iteration complexity bounds \eqref{rAdaAPG-compl-O} in the cases $\rho=2$ and $\rho>2$ match the lower bounds \eqref{low-compl-grad-HEB} up to a logarithmic factor.
Therefore, $\rAdaAPG$ achieves the near optimality for the class $\mathcal{F}(x_0,R,L,\kappa,\rho)$.

Note that
the choice \eqref{sigma0-choice-HEB} of $\sigmainit$ to obtain the above iteration complexity is independent of the target tolerance $\ep$.
Therefore, if we consider $\rAdaAPG$ with $\ep=0$ (recall that $\ep$ only affects the stopping criterion), it yields a convergent algorithm with near optimality.
In summary,
to find an approximate solution $x$ with $\norm{g_L(x)}\leq \ep$, the proposed method $\rAdaAPG$ is able to attain the following iteration complexity (and the rate of convergence).
\begin{itemize}
\item Case $\rho=1$. The algorithm finds an optimal solution with a finite iteration complexity, i.e., the iteration complexity is independent of $\ep$ (if $\ep$ is sufficiently small).
\item Case $\rho\in ]1,2[$. The iteration complexity is $O(\log\log\ep^{-1})$ (superlinear convergence).
\item Case $\rho=2$. The iteration complexity is $O(\log\ep^{-1})$ (linear convergence).
\item Case $\rho>2$. The iteration complexity is $O(\ep^{-\frac{\rho-2}{2(\rho-1)}}\log\ep^{-1})$ (sublinear convergence).
\end{itemize}

The algorithm $\rAdaAPG$ with $\ep=0$ can also provide an iteration complexity result with respect to the measure $\varphi(\cdot)-\varphi^*$.
As we prove the inequality \eqref{rAdaAPG-obj-bound}, the following relation holds if $\rho \ne 1$.
\begin{equation}\label{objval-gmapnorm-rel}
\varphi(\outerXplus{t})-\varphi^* \leq \frac{1}{\kappa^{\frac{1}{\rho-1}}}\left(\frac{L_f}{L_{\min}}+1\right)^{\frac{\rho}{\rho-1}}\norm{g_{\outerM{t}}(\outerXinit{t})}^{\frac{\rho}{\rho-1}}.
\end{equation}
This means that, given $\delta>0$, we have the following implication:
\begin{equation*}
\norm{g_{\outerM{t}}(\outerXinit{t})} \leq \ep:=\kappa^{\frac{1}{\rho}}\left(\frac{L_f}{L_{\min}}+1\right)^{-1} \delta^{\frac{\rho-1}{\rho}} ~~\Longrightarrow~~
\varphi(\outerXplus{t})-\varphi^* \leq \delta.
\end{equation*}
Substituting this $\ep$ to our complexity bound \eqref{rAdaAPG-compl-O}, we also obtain an iteration complexity bound under the measure $\varphi(\cdot)-\varphi^*$, which is nearly optimal in view of the lower complexity bound \eqref{low-compl-obj-HEB}.
Although it enjoys an adaptive and nearly optimal convergence, the proposed method does not provide a stopping criterion for the measure $\varphi(\cdot)-\varphi^*$ since the right hand side of \eqref{objval-gmapnorm-rel} is not verifiable unless we know $\kappa$ and $\rho$.

\subsubsection{Proof of the Main Results}\label{sec:proof-main-result}

Here, we complete the proofs of Theorem~\ref{thm:rAdaAPG} and Corollary~\ref{cor:rAdaAPG}.
We prepare some lemmas below.

\begin{lemma}\label{lem:rAdaAPG-base}
Assume that the H\"olderian error bound condition \eqref{HEB-review} holds.
In the execution of $\rAdaAPG$, the following assertions hold.
\begin{enumerate}[(i)]
\item $\outerL{t},\outerM{t} \in [L_{\min},\gamma_{\rm inc}L_f]$ for all $t\geq 0$.
\item $\varphi(\outerXplus{t+1}) \leq \varphi(\outerXplus{t}) \leq \cdots \leq \varphi(\outerXplus{0}) \leq \varphi(\xinit)$ for all $t\geq 0$.
\item $t \leq \log_{1/\theta}\dfrac{\norm{g_{\outerM{0}}(\xinit)}}{\norm{g_{\outerM{t}}(\outerXinit{t})}}$ for each $t\geq 0$.
\item Whenever $\outerXplus{t} \not\in X^*$, we have
\begin{equation}\label{rAdaAPG-dist-bound}
\dist(\outerXplus{t},X^*)^{\rho-1} \leq \frac{1}{\kappa}\left(\frac{L_f}{L_{\min}}+1\right)\norm{g_{\outerM{t}}(\outerXinit{t})},
\end{equation}
\begin{equation}\label{rAdaAPG-obj-bound}
(\varphi(\outerXplus{t})-\varphi^*)^{\rho-1} \leq \frac{1}{\kappa}\left(\frac{L_f}{L_{\min}}+1\right)^\rho\norm{g_{\outerM{t}}(\outerXinit{t})}^\rho.
\end{equation}
\end{enumerate}
\end{lemma}
\begin{proof}
(i)
Since $\Linit \in [L_{\min}, \gamma_{\rm inc}L_f]$, it follows $\outerL{0},\outerM{0}\in [L_{\min}, \gamma_{\rm inc}L_f]$ by Lemma~\ref{lem:PG-iter}~(ii).
Then, using Proposition~\ref{prop:APG} (i) inductively, we obtain (i).

(ii)
For $t=0$, we have $\varphi(\outerXplus{0})=\varphi(T_{\outerM{0}}(\xinit))\leq \varphi(\xinit)$ by Lemma~\ref{lem:PG-iter} (i).
Moreover, the criterion \eqref{APG-iter-test-c} and Proposition~\ref{prop:APG}~(iv), applied to the subroutine $\AdaAPG$ in $\rAdaAPG$, yield
${\varphi(\outerXplus{t+1}) \leq \varphi(\outerXinit{t+1})}$ and $\varphi(\outerXinit{t+1}) \leq \varphi(\outerXplus{t})$, respectively. This shows $\varphi(\outerXplus{t+1}) \leq \varphi(\outerXplus{t})$.

(iii)
Since the recurrence
$
\ep^{(t)}/\theta=\norm{g_{\outerM{t}}(\outerXinit{t})} \leq \ep^{(t-1)}
$
holds for $t\geq 1$, we have
$
\ep^{(0)} \geq \ep^{(t)}/\theta^t
$
for each $t\geq 0$.
Therefore, we conclude
$
t \leq \log_{1/\theta}\frac{\ep^{(0)}}{\ep^{(t)}} = \log_{1/\theta}\frac{\norm{g_{\outerM{0}}(\xinit)}}{\norm{g_{\outerM{t}}(\outerXinit{t})}}.
$

(iv)
Since $\outerXplus{t}=T_{\outerM{t}}(\outerXinit{t})$, Lemma~\ref{lem:gmap} (iii) implies that $g_t:=\nabla{f}(\outerXplus{t})-\nabla{f}(\outerXinit{t})+g_{\outerM{t}}(\outerXinit{t})$ belongs to $\partial{\varphi}(\outerXplus{t})$.
Then,
Lemma~\ref{lem:HEB} shows (note that $\outerXplus{t}$ belongs to $\lev_\varphi(\varphi(\xinit))\setminus X^*$ by (ii))
\begin{align*}
\kappa\dist(\outerXplus{t},X^*)^{\rho-1} &\leq \norm{g_t}
\leq \left(\frac{L_f}{\outerM{t}}+1\right)\norm{g_{\outerM{t}}(\outerXinit{t})}
\leq \left(\frac{L_f}{L_{\min}}+1\right)\norm{g_{\outerM{t}}(\outerXinit{t})},
\end{align*}
where the second inequality is due to Lemma~\ref{lem:gmap} (iii) and the last follows by (i). This shows the assertion \eqref{rAdaAPG-dist-bound}. Similarly, \eqref{rAdaAPG-obj-bound} can be obtained using the latter inequality of \eqref{Loj-ineq}.
\end{proof}

The following lemma plays an essential role to derive our iteration complexity results.

\begin{lemma}\label{lem:rAdaAPG-compl}
Let $N^{(t)}$ be the number of executions of $\APGIter$ at the $t$-th outer loop of $\rAdaAPG$.
Assume that, for some $T_*\geq 0$, $\sigma_*>0$, and $\ep_*>0$, we have
$$
\outersigma{t+1}\geq \sigma_*\quad (t=0,\ldots,T_*)
\quad
\text{and}
\quad
\norm{g_{\outerM{T_*}}(\outerXinit{T_*})} \geq \ep_*.
$$
Then, under the definition \eqref{N-def} of $N(\cdot,\cdot,\cdot)$, the following inequality holds.
$$
\sum_{t=0}^{T_*}N^{(t)} \leq N(\ep_*,\sigma_*,C) \leq N(\ep_*,\sigma_*,C_*),
$$
where
$$
C := \sum_{t=0}^{T_*}\sqrt{\frac{1}{\outersigma{t+1}}}
\leq C_*:= \sqrt{\frac{1}{\sigma_*}}\left(1+\log_{1/\theta}\frac{\norm{g_{\outerM{0}}(\xinit)}}{\ep_*}\right).
$$
\end{lemma}
\begin{proof}
Since $\outersigma{t+1}\geq \sigma_*$ for each $t=0,\ldots,T_*$, Theorem~\ref{thm:AdaAPG} (ii) gives the following bound for $t=0,\ldots,T_*$:
\begin{align*}
N^{(t)}
&\leq 
\left(1+\log_{\gamma_{\rm reg}}\frac{\outersigma{t}}{\outersigma{t+1}}\right)\left(2+\log\frac{\gamma_{\rm inc}L_f+\outersigma{t+1}}{\beta\outersigma{t+1}}\right)\\
&\quad +
\frac{\sqrt{2\gamma_{\rm inc}L_f}}{\sqrt{\gamma_{\rm reg}}-1}\left[\sqrt{\gamma_{\rm reg}}\sqrt{\frac{1}{\outersigma{t+1}}} - \sqrt{\frac{1}{\outersigma{t}}}\right]\log\frac{\gamma_{\rm inc}L_f+\outersigma{t+1}}{\beta\outersigma{t+1}}\\
&\leq
\left(1+\log_{\gamma_{\rm reg}}\frac{\outersigma{t}}{\outersigma{t+1}}\right)\left(2+\log\frac{\gamma_{\rm inc}L_f+\sigma_*}{\beta\sigma_*}\right)\\
&\quad +
\frac{\sqrt{2\gamma_{\rm inc}L_f}}{\sqrt{\gamma_{\rm reg}}-1}\left[\sqrt{\gamma_{\rm reg}}\sqrt{\frac{1}{\outersigma{t+1}}} - \sqrt{\frac{1}{\outersigma{t}}}\right]\log\frac{\gamma_{\rm inc}L_f+\sigma_*}{\beta\sigma_*}.
\end{align*}
By Lemma~\ref{lem:rAdaAPG-base} (iii), $T_*$ is bounded by
$$
T_*\leq \log_{1/\theta}\frac{\norm{g_{\outerM{0}}(\xinit)}}{\norm{g_{\outerM{T_*}}(\outerXinit{T_*})}} \leq \log_{1/\theta}\frac{\norm{g_{\outerM{0}}(\xinit)}}{\ep_*}.
$$
To estimate the sum $\sum_{t=0}^{T_*}N^{(t)}$, note that
$$
\sum_{t=0}^{T_*}\left(1+\log_{\gamma_{\rm reg}}\frac{\outersigma{t}}{\outersigma{t+1}}\right) = 1+T_*+\log_{\gamma_{\rm reg}}\frac{\outersigma{0}}{\outersigma{T_*+1}} \leq 1+ \log_{1/\theta}\frac{\norm{g_{\outerM{0}}(\xinit)}}{\ep_*} + \log_{\gamma_{\rm reg}}\frac{\outersigma{0}}{\sigma_*},
$$
and
\begin{align*}
\sum_{t=0}^{T_*}
\left[\sqrt{\gamma_{\rm reg}}\sqrt{\frac{1}{\outersigma{t+1}}} - \sqrt{\frac{1}{\outersigma{t}}}\right]
&=
(\sqrt{\gamma_{\rm reg}}-1)\sum_{t=0}^{T_*}\sqrt{\frac{1}{\outersigma{t+1}}}
+\sum_{t=0}^{T_*}
\left[\sqrt{\frac{1}{\outersigma{t+1}}} - \sqrt{\frac{1}{\outersigma{t}}}\right]\\
&=
(\sqrt{\gamma_{\rm reg}}-1)\sum_{t=0}^{T_*}\sqrt{\frac{1}{\outersigma{t+1}}}
+
\sqrt{\frac{1}{\outersigma{T_*+1}}} - \sqrt{\frac{1}{\outersigma{0}}}.
\end{align*}
Therefore, $\sum_{t=0}^{T_*}N^{(t)}$ is bounded by
\begin{align*}
\sum_{t=0}^{T_*}N^{(t)}
&\leq \left(1+ \log_{1/\theta}\frac{\norm{g_{\outerM{0}}(\xinit)}}{\ep_*} + \log_{\gamma_{\rm reg}}\frac{\outersigma{0}}{\sigma_*}\right)\left(2+\log\frac{\gamma_{\rm inc}L_f+\sigma_*}{\beta\sigma_*}\right)\\
&\quad +
\frac{\sqrt{2\gamma_{\rm inc}L_f}}{\sqrt{\gamma_{\rm reg}}-1}\left[
\sqrt{\frac{1}{\sigma_*}} - \sqrt{\frac{1}{\outersigma{0}}}\right]\log\frac{\gamma_{\rm inc}L_f+\sigma_*}{\beta\sigma_*} \\
&\quad + \sqrt{2\gamma_{\rm inc}L_f} \sum_{t=0}^{T_*}\sqrt{\frac{1}{\outersigma{t+1}}} \log\frac{\gamma_{\rm inc}L_f+\sigma_*}{\beta\sigma_*}\\
&=N(\ep_*,\sigma_*,C).
\end{align*}
Finally, $C$ has the following bound:
$$
C = \sum_{t=0}^{T_*}\sqrt{\frac{1}{\outersigma{t+1}}}
\leq \sqrt{\frac{1}{\sigma_*}} (1+T_*)\leq
\sqrt{\frac{1}{\sigma_*}}\left(1+\log_{1/\theta}\frac{\norm{g_{\outerM{0}}(\xinit)}}{\ep_*}\right)=C_*,
$$
which also concludes $N(\ep_*,\sigma_*,C) \leq N(\ep_*,\sigma_*,C_*)$.
\end{proof}

In view of this lemma,
the complexity analysis boils down to analyze lower bounds of $\outersigma{t}$ as we discuss next.

\begin{lemma}\label{lem:rAdaAPG-sigma}
Assume that the H\"olderian error bound condition \eqref{HEB-review} holds.
Suppose that $\rAdaAPG$ terminated with the stopping criterion at the $(T+1)$-th outer loop for some $T\geq 0$.
Let $\bar\sigma$ be defined by \eqref{bar-sigma-def}.
\begin{enumerate}[(i)]
\item
For each $t=0,\ldots,T$, we have
$$
\outersigma{t+1}
\left\{
\begin{array}{ll}
=\outersigma{t}&(\outersigma{t}\leq \sigma(\outerXplus{t},\ep^{(t)})),\\
\geq\sigma(\outerXplus{t},\ep^{(t)})/\gamma_{\rm reg}&(\text{otherwise}),
\end{array}
\right.
$$
where $\sigma(\cdot,\cdot)$ is defined by \eqref{sigma-th}. 

\item It follows that $\sigma(\outerXplus{t},\ep^{(t)}) \geq \bar\sigma$ for all $t=0,\ldots,T$.
When $\rho \geq 2$, we further obtain
\begin{equation}\label{sigma-th-t}
\sigma(\outerXplus{t},\ep^{(t)}) \geq \frac{\theta^{\frac{1}{\rho-1}}}{1+\sqrt{2}\beta}\cdot
\kappa^{\frac{1}{\rho-1}}\left(\frac{L_f}{L_{\min}}+1\right)^{-\frac{1}{\rho-1}}
(\ep^{(t)})^{\frac{\rho-2}{\rho-1}}
 \geq \bar\sigma,
\end{equation}
for each $t=0,\ldots,T$.
\item
For each $t=0,\ldots,T+1$, we obtain
$$
\outersigma{t} \left\{
\begin{array}{ll}
=\sigmainit & (\sigmainit \leq \bar\sigma),\\
\geq\bar\sigma/\gamma_{\rm reg} & (\text{otherwise}).
\end{array}
\right.
$$
\end{enumerate}
\end{lemma}

\begin{proof}
(i)
In the case $\outersigma{t} \leq \sigma(\outerXplus{t},\ep^{(t)})$, Theorem~\ref{thm:AdaAPG} (i) implies that the subroutine $\AdaAPG$ at the $t$-th outer loop must terminate at the first loop $j=0$ so that $\outersigma{t+1}$ is defined by $\outersigma{t+1}=\outersigma{t}$.
On the other hand, consider the case $\outersigma{t} > \sigma(\outerXplus{t},\ep^{(t)})$.
When the subroutine $\AdaAPG$ at the $t$-th outer loop terminates at the first loop $j=0$, then it is clear that
$\outersigma{t+1}=\outersigma{t} > \sigma(\outerXplus{t},\ep^{(t)}) \geq \sigma(\outerXplus{t},\ep^{(t)})/\gamma_{\rm reg}.$
In the other case, the latter assertion of Theorem~\ref{thm:AdaAPG}~(i) implies $\outersigma{t+1}\geq \sigma(\outerXplus{t},\ep^{(t)})/\gamma_{\rm reg}.$

(ii)
We may assume $\outerXplus{t} \not\in X^*$, which allows us to apply Lemma~\ref{lem:rAdaAPG-base} (iv) (Note that the assertion is clear if $\outerXplus{t} \in X^*$ since then $\sigma(\outerXplus{t},\ep^{(t)})=+\infty$).
In the case $\rho\geq 2$, using \eqref{rAdaAPG-dist-bound} implies
$$
\sigma(\outerXplus{t},\ep^{(t)})
= \frac{\theta\norm{g_{\outerM{t}}(\outerXinit{t})}}{(1+\sqrt{2}\beta)\dist(\outerXplus{t},X^*)}
\geq
\frac{\theta}{1+\sqrt{2}\beta}\cdot
\kappa^{\frac{1}{\rho-1}}\left(\frac{L_f}{L_{\min}}+1\right)^{-\frac{1}{\rho-1}}
\norm{g_{\outerM{t}}(\outerXinit{t})}^{\frac{\rho-2}{\rho-1}}.
$$
Since $\ep^{(t)}\equiv \theta\|g_{\outerM{t}}(\outerXinit{t})\|$ and $\|g_{\outerM{t}}(\outerXinit{t})\| \geq \ep$ (by the definition of $T$), we obtain \eqref{sigma-th-t}.

If $\rho\in [1,2[$, on the other hand,
since $\varphi(\outerXplus{t})\leq \varphi(\xinit)$ by Lemma~\ref{lem:rAdaAPG-base} (ii), the inequalities \eqref{HEB-review} and \eqref{rAdaAPG-obj-bound} imply (remark that $\frac{2-\rho}{\rho}>0$)
\begin{align*}
\dist(\outerXplus{t},X^*) &\leq \frac{1}{\kappa^{\frac{1}{\rho}}}(\varphi(\outerXplus{t})-\varphi^*)^{\frac{1}{\rho}}
=
\frac{1}{\kappa^{\frac{1}{\rho}}}(\varphi(\outerXplus{t})-\varphi^*)^{\frac{\rho-1}{\rho}}(\varphi(\outerXplus{t})-\varphi^*)^{\frac{2-\rho}{\rho}}\\
&\leq
\frac{1}{\kappa^{\frac{2}{\rho}}}\left(\frac{L_f}{L_{\min}}+1\right)\norm{g_{\outerM{t}}(\outerXinit{t})}(\varphi(\outerXplus{t})-\varphi^*)^{\frac{2-\rho}{\rho}}\\
&\leq
\frac{1}{\kappa^{\frac{2}{\rho}}}\left(\frac{L_f}{L_{\min}}+1\right)\norm{g_{\outerM{t}}(\outerXinit{t})}(\varphi(\xinit)-\varphi^*)^{\frac{2-\rho}{\rho}}.
\end{align*}
Therefore, we conclude that
$$
\sigma(\outerXplus{t},\ep^{(t)})
= \frac{\theta\norm{g_{\outerM{t}}(\outerXinit{t})}}{(1+\sqrt{2}\beta)\dist(\outerXplus{t},X^*)}
\geq
\frac{\theta}{1+\sqrt{2}\beta}
\kappa^{\frac{2}{\rho}}\left(\frac{L_f}{L_{\min}}+1\right)^{-1}(\varphi(\xinit)-\varphi^*)^{-\frac{2-\rho}{\rho}}
=\bar\sigma.
$$
This proves the assertion (ii).

(iii)
In the case $\sigmainit\leq \bar\sigma$, (ii) implies $\outersigma{0} \leq \sigma(\outerXplus{0},\ep^{(0)})$. Then, using (i), we have $\outersigma{1}=\outersigma{0}$ and also $\outersigma{1} \leq \bar\sigma$.
Continuing this argument inductively, we conclude that $\outersigma{t}=\outersigma{t-1}=\cdots=\outersigma{0}$ for all $t$.

In the case $\sigmainit \geq \bar\sigma$, on the other hand, let us show $\outersigma{t} \geq \bar\sigma/\gamma_{\rm reg}~(t=0,\ldots,T+1)$ by induction. The assertion for $t=0$ is clear since $\sigmainit \geq \bar\sigma \geq \bar\sigma/\gamma_{\rm reg}$.
Assume that $\outersigma{t}\geq \bar\sigma/\gamma_{\rm reg}$ holds for some $t\geq 0$. By (i) and (ii), we have
$
\outersigma{t+1} \geq \min\{\outersigma{t}, \sigma(\outerXplus{t},\ep^{(t)})/\gamma_{\rm reg}\} \geq \min\{\bar\sigma/\gamma_{\rm reg}, \bar\sigma/\gamma_{\rm reg}\} =\bar\sigma/\gamma_{\rm reg}.
$
This completes the proof of (iii).
\end{proof}

In order to provide more accurate complexity analysis, we prove bounds of $\outersigma{t}$ specialized to the case $\rho > 2$.

\begin{lemma}\label{lem:rAdaAPG-sigma-2}
Assume that the H\"olderian error bound condition \eqref{HEB-review} holds with $\rho >2$.
Suppose that $\rAdaAPG$ terminated with the stopping criterion at the $(T+1)$-th outer loop for some $T\geq 0$.
If $\sigmainit \geq \bar\sigma$ holds for $\bar\sigma$  defined by \eqref{bar-sigma-def},
then there exists $t_0 \in \{0,\ldots,T+1\}$ such that the following conditions hold, where
$
\sigma_*^{(0)}:=\frac{\theta}{1+\sqrt{2}\beta}\cdot
\kappa^{\frac{1}{\rho-1}}\left(\frac{L_f}{L_{\min}}+1\right)^{-\frac{1}{\rho-1}}
\|g_{\outerM{0}}(\xinit)\|^{\frac{\rho-2}{\rho-1}}.
$
\begin{align}
\text{(i)\qquad} & \outersigma{0}=\cdots=\outersigma{t_0}. \label{sigma-bound-rho>2-1}\\
\text{(ii)\qquad} & \outersigma{t+1} \geq \theta^{-\frac{\rho-2}{\rho-1}(T-t)} \bar\sigma/\gamma_{\rm reg},\quad t=t_0,\ldots,T.\label{sigma-bound-rho>2-2}\\
\text{(iii)\qquad} & t_0 \leq 1+\frac{\rho-1}{\rho-2}\log_{1/\theta}\frac{\sigma_*^{(0)}}{\min(\sigmainit,\sigma_*^{(0)})}.\label{sigma-bound-rho>2-3}\\
\text{(iv)\qquad} & \theta^{\frac{\rho-2}{\rho-1}(T-t_0)} \geq {\bar\sigma}/{\outersigma{0}}.\label{sigma-bound-rho>2-4}
\end{align}

\end{lemma}
\begin{proof}
Define
$$
\sigma_*^{(t)}:=\frac{\theta^{\frac{1}{\rho-1}}}{1+\sqrt{2}\beta}\cdot
\kappa^{\frac{1}{\rho-1}}\left(\frac{L_f}{L_{\min}}+1\right)^{-\frac{1}{\rho-1}}
(\ep^{(t)})^{\frac{\rho-2}{\rho-1}},\quad t \geq 0.
$$
Then, Lemma~\ref{lem:rAdaAPG-sigma} (ii) can be written as
\begin{equation}\label{simga_*^t-bound}
\sigma(\outerXplus{t},\ep^{(t)}) \geq \sigma_*^{(t)} \geq \bar\sigma,\quad t=0,\ldots,T.
\end{equation}
For simplicity, denote
$$
\omega:=\frac{\rho-2}{\rho-1} \in ]0,1[,\quad
c:=\frac{\theta^{\frac{1}{\rho-1}}}{1+\sqrt{2}\beta}\cdot
\kappa^{\frac{1}{\rho-1}}\left(\frac{L_f}{L_{\min}}+1\right)^{-\frac{1}{\rho-1}},
$$
so that we have
$$
\bar\sigma = c(\theta\ep)^{\omega},\quad
\sigma_*^{(t)}=c(\ep^{(t)})^{\omega},\quad t\geq 0.
$$
Note that $\{\sigma_*^{(t)}\}$ is non-increasing;
in fact, the relation $\ep^{(t+1)}\leq \theta\ep^{(t)}$ implies
\begin{equation}\label{sigma-*-mono}
\sigma_*^{(t+1)}=c(\ep^{(t+1)})^\omega\leq c\theta^\omega(\ep^{(t)})^\omega = \theta^\omega \sigma_*^{(t)} \leq \sigma_*^{(t)}.
\end{equation}

Let $t_0$ be the smallest integer in $\{0,\ldots,T+1\}$ such that $\sigmainit \geq \sigma_*^{(t_0)}$.
Remark that, by the definition of $T$, we have $\ep^{(T+1)}=\theta\|g_{\outerM{T+1}}(\outerXinit{T+1})\| \leq \theta\ep$.
This implies that
$
\sigmainit \geq \bar\sigma= c(\theta\ep)^\omega \geq c(\ep^{(T+1)})^\omega = \sigma_*^{(T+1)}.
$
Therefore, $t_0$ is well-defined.

(i)
By the definition of $t_0$ and \eqref{simga_*^t-bound}, we have
$\sigmainit < \sigma_*^{(t)} \leq \sigma(\outerXplus{t},\ep^{(t)})$ ($0\leq \forall t < t_0$).
Therefore, by induction, we obtain $\outersigma{0}=\outersigma{1}=\cdots=\outersigma{t_0}$ due to Lemma~\ref{lem:rAdaAPG-sigma}~(i).

(ii)
Let us show $\outersigma{t+1}\geq \sigma_*^{(t)}/\gamma_{\rm reg}~(t_0\leq t\leq T)$.
To prove this, we verify $\outersigma{t}\geq \sigma_*^{(t)}/\gamma_{\rm reg}$ for ${t= t_0,\ldots,T+1}$ by induction. Note that $\outersigma{t_0}=\outersigma{0} \geq  \sigma_*^{(t_0)} > \sigma_*^{(t_0)}/\gamma_{\rm reg}$ holds by (i) and the definition of $t_0$. Now under the hypothesis $\outersigma{t} \geq  \sigma_*^{(t)}/\gamma_{\rm reg}$ for $t$ with $t_0\leq t \leq T$, Lemma~\ref{lem:rAdaAPG-sigma}~(i) and \eqref{simga_*^t-bound} imply
$$\outersigma{t+1} \geq \min\{\outersigma{t},\sigma(\outerXplus{t},\ep^{(t)})/\gamma_{\rm reg}\} \geq \min\{\sigma_*^{(t)}/\gamma_{\rm reg},\sigma_*^{(t)}/\gamma_{\rm reg}\} = \sigma_*^{(t)}/\gamma_{\rm reg} \geq \sigma_*^{(t+1)}/\gamma_{\rm reg}.$$
Therefore, this completes the induction; in addition, the above inequality involves the desired inequality $\outersigma{t+1}\geq \sigma_*^{(t)}/\gamma_{\rm reg}~(t_0\leq t\leq T)$.
This yields (ii) combined with \eqref{simga_*^t-bound} and \eqref{sigma-*-mono}:
$$
\outersigma{t+1} \geq \sigma_*^{(t)}/\gamma_{\rm reg} \geq \theta^{-\omega(T-t)}\sigma_*^{(T)}/\gamma_{\rm reg} \geq \theta^{-\omega(T-t)} \bar\sigma/\gamma_{\rm reg},\quad t=t_0,\ldots,T.
$$

(iii)
If $t_0=0$, then
(iii) is trivial since
$\sigmainit \geq\sigma_*^{(t_0)}=\sigma_*^{(0)}$. If $t_0 > 0$, then the definition of $t_0$ and using \eqref{sigma-*-mono} imply
$\sigmainit < \sigma_*^{(t_0-1)} \leq \theta^{\omega(t_0-1)}\sigma_*^{(0)},$
which yields $t_0 \leq 1+\frac{1}{\omega}\log_{1/\theta}\frac{\sigma_*^{(0)}}{\sigmainit}$.

(iv)
By \eqref{simga_*^t-bound}, \eqref{sigma-*-mono}, and the definition of $t_0$, remark that
$
\bar\sigma \leq \sigma_*^{(T)} \leq \theta^{\omega(T-t_0)}\sigma_*^{(t_0)}\leq \theta^{\omega(T-t_0)}\sigmainit.
$
Hence, 
$
\theta^{\omega(T-t_0)} \geq {\bar\sigma}/{\outersigma{0}}
$
holds.
\end{proof}

Finally, we present the proofs of Theorem~\ref{thm:rAdaAPG} and Corollary~\ref{cor:rAdaAPG}.

\begin{proof}[\underline{Proof of Theorem~\ref{thm:rAdaAPG}}]
We may assume that $\norm{g_{\outerM{0}}(\xinit)} > \ep$ since $N=0$ on the other case.
Suppose that $\rAdaAPG$ terminated with the stopping criterion at the $(T+1)$-th outer loop for some $T\geq 0$.
Denote by $N^{(t)}$ the number of executions of $\APGIter$ at the $t$-th outer loop so that $N=\sum_{t=0}^T N^{(t)}$.
Let $\bar\sigma$ be defined by \eqref{bar-sigma-def}.

By the definition of $T$, we have
\begin{equation}\label{pr:gradmap-lb-ep}
\norm{g_{\outerM{t}}(\outerXinit{t})} > \ep,\quad t=0,\ldots,T.
\end{equation}
Moreover, by the definition of $\sigma_*$ in \eqref{simga*-def},
using Lemma~\ref{lem:rAdaAPG-sigma} (iii) implies
$$
\outersigma{t} \geq
\sigma_*,\quad t=0,\ldots,T+1.
$$
Therefore, applying Lemma~\ref{lem:rAdaAPG-compl} with $T_*=T$, we obtain the assertion
\begin{equation}\label{N-ub-basic-HEB}
N \leq N(\ep,\sigma_*,C) \quad \text{with}\quad
C=\sqrt{\frac{1}{\sigma_*}}\left(1+\log_{1/\theta}\frac{\norm{g_{\outerM{0}}(\xinit)}}{\ep}\right).
\end{equation}
For the case $\rho=2$, \eqref{N-ub-basic-HEB} proves the assertion (i).
We discuss the other cases to improve \eqref{N-ub-basic-HEB}.

\fbox{(ii) Case $\rho>2$.}
If $\sigmainit < \bar\sigma$, then $\sigma_*$ is defined as $\sigma_*=\sigmainit$. 
Therefore, the latter bound of (ii) is obtained by \eqref{N-ub-basic-HEB}.

Now consider the case $\sigmainit \geq \bar\sigma$.
Then, there exists $t_0 \in \{0,\ldots,T+1\}$ satisfying the conditions in Lemma~\ref{lem:rAdaAPG-sigma-2}.
By \eqref{sigma-bound-rho>2-4}, remark that
\begin{equation}\label{rAdaAPG-compl-p>2-pf-sum}
\sum_{t=t_0}^{T}\sqrt{\theta^{\frac{\rho-2}{\rho-1}(T-t)}}
=
\sum_{i=0}^{T-t_0}\sqrt{\theta^{\frac{\rho-2}{\rho-1}}}^{\,i}
=
\frac{1-\sqrt{\theta^{\frac{\rho-2}{\rho-1}}}^{T-t_0+1}}{1-\sqrt{\theta^{\frac{\rho-2}{\rho-1}}}}
\leq
\frac{1-\sqrt{\theta^{\frac{\rho-2}{\rho-1}}\cdot\dfrac{\bar\sigma}{\outersigma{0}}}}{1-\sqrt{\theta^{\frac{\rho-2}{\rho-1}}}}.
\end{equation}
Therefore, we conclude that
\begin{align*}
\sum_{t=0}^{T}\sqrt{\frac{1}{\outersigma{t+1}}}
&=
\sum_{t=0}^{t_0-1} \sqrt{\frac{1}{\outersigma{t+1}}}+
\sum_{t=t_0}^{T}\sqrt{\frac{1}{\outersigma{t+1}}}
\\
&\leq
t_0 \sqrt{\frac{1}{\outersigma{0}}} +
\sqrt{\frac{\gamma_{\rm reg}}{\bar\sigma}}\sum_{t=t_0}^{T}\sqrt{\theta^{{\frac{\rho-2}{\rho-1}}(T-t)}} & (\text{by } \eqref{sigma-bound-rho>2-1} \text{ and } \eqref{sigma-bound-rho>2-2})\\
&\leq
\sqrt{\frac{1}{\outersigma{0}}}\left(1+\frac{\rho-1}{\rho-2}\log_{1/\theta}\frac{\sigma_*^{(0)}}{\min(\sigma_*^{(0)},\outersigma{0})}\right)\\
&\quad~+
\frac{\sqrt{\gamma_{\rm reg}}}{1-\sqrt{\theta^{\frac{\rho-2}{\rho-1}}}}
\left(
\sqrt{\frac{1}{\bar\sigma}}
-
\sqrt{\theta^{\frac{\rho-2}{\rho-1}
}}
\sqrt{\frac{1}{\outersigma{0}}}
\right) & (\text{by } \eqref{sigma-bound-rho>2-3} \text{ and } \eqref{rAdaAPG-compl-p>2-pf-sum})\\
&=:C.
\end{align*}
With this definition of $C$, Lemma~\ref{lem:rAdaAPG-compl} gives $N \leq N(\ep,\sigma_*,C)$.

\fbox{(iii) Case $\rho \in ]1,2[$.}
If $\ep \geq \ep_*$, then \eqref{N-ub-basic-HEB} gives our assertion because the second term of \eqref{N-bd-superlinear} vanishes.
Suppose, on the other hand, that $\ep < \ep_*$. 
Denote
$\xi_{t}:=\norm{g_{\outerM{t}}(\outerXinit{t})}$
and 
let $T_*\geq 0$ be the smallest integer such that $\xi_{T_*} \leq \ep_*$.
Then, since $\xi_{T_*-1} > \ep_*$ holds, Lemma~\ref{lem:rAdaAPG-compl} shows that
$$
\sum_{t=0}^{T_*-1} N^{(t)} \leq N(\ep_*,\sigma_*,C)\quad \text{with}\quad C=\sqrt{\frac{1}{\sigma_*}}\left( 1 + \log_{1/\theta}\frac{\norm{g_{\outerM{0}}(\xinit)}}{\ep_*} \right).
$$
Note that, under the convention $\sum_{t=0}^{-1}(\cdot)=0$, this inequality also holds if $T_*=0$ since then $\xi_0\leq \ep_*$ and $N(\ep_*,\sigma_*,C)=0$.

It remains to observe $\sum_{t=T_*}^T N^{(t)}$.
Take $t\in \{T_*,\ldots,T\}$. We shall prove $N^{(t)}=1$.
In the $t$-th outer loop, consider the first iteration of the subroutine $\AdaAPG(\outerXplus{t},\outerL{t},\outersigma{t},\ep^{(t)})$, which executes
$$
\{x_1,\psi_1,M_0,L_1,A_1\} \leftarrow \APGIter_{\outersigma{t}}(\outerXplus{t},\psi_0,\outerL{t},A_0),
$$
where $\psi_0(x)=\frac{1}{2}\|x-\outerXplus{t}\|^2$ and $A_0=0$.
Then, Proposition~\ref{prop:APG} (v) implies
$$
\norm{g_{M_{0}}(x_1)} \leq \left(2\sqrt{\frac{M_{0}+\outersigma{t}}{A_1}} + \outersigma{t}\right)\dist(\outerXplus{t},X^*).
$$
Moreover, according to the equation at Line~\ref{a-quad-eq} in Algorithm~\ref{APG-iter}, $A_1$ can be calculated as $A_1=2/M_0$. Now remark that, using $M_0 \leq \gamma_{\rm inc}L_f$ (Proposition~\ref{prop:APG}~(i)), we have
\begin{align*}
2\sqrt{\frac{M_{0}+\outersigma{t}}{A_1}} + \outersigma{t} &= 2\sqrt{\frac{M_0(M_0+\outersigma{t})}{2}} + \outersigma{t}
\leq 2\sqrt{\frac{(M_0+\outersigma{t})^2}{2}} + (M_0+\outersigma{t}) \\&= (1+\sqrt{2})(M_0+\outersigma{t}) \leq (1+\sqrt{2})(\gamma_{\rm inc}L_f + \outersigma{t}).
\end{align*}
Combining them and using \eqref{rAdaAPG-dist-bound}\footnote{Although \eqref{rAdaAPG-dist-bound} is asserted in the case $\outerXplus{t} \not\in X^*$, it trivially holds if $\outerXplus{t} \in X^*$ \emph{unless} $\rho = 1$.}, we conclude that
\begin{align}
\norm{g_{M_{0}}(x_1)}
&\leq
(1+\sqrt{2})(\gamma_{\rm inc}L_f+\outersigma{0})\dist(\outerXplus{t},X^*) \nonumber\\
&\leq
(1+\sqrt{2})(\gamma_{\rm inc}L_f+\outersigma{0})
\left(\frac{1}{\kappa}\right)^{\frac{1}{\rho-1}} \left(\frac{L_f}{L_{\min}}+1\right)^{\frac{1}{\rho-1}}\xi_t^{\frac{1}{\rho-1}}
\nonumber\\
&= \theta\ep_*^{\frac{\rho-2}{\rho-1}}\xi_t^{\frac{1}{\rho-1}}
\label{superlinear-rec}\\
&\leq \theta\xi_t^{\frac{\rho-2}{\rho-1}}\xi_t^{\frac{1}{\rho-1}}=\theta\xi_t = \ep^{(t)},
\nonumber
\end{align}
where the last inequality is due to $\xi_t \leq \ep_*$ for $t \geq T_*$ (and remark $\frac{\rho-2}{\rho-1} < 0$).
This shows $N^{(t)}=1$ and \eqref{superlinear-rec} yields the recurrence
$
\xi_{t+1} \leq \theta\ep_*^{\frac{\rho-2}{\rho-1}}\xi_t^{\frac{1}{\rho-1}}
$,
for each $t=T_*,\ldots,T$.
Since $\frac{1}{\rho-1} > 1$, it reduces $\xi_t$ superlinearly. In particular, solving this recurrence implies
$$
\log \frac{\ep_*}{\theta^{\frac{\rho-1}{2-\rho}}\xi_T} \geq \left(\frac{1}{\rho-1}\right)^{T-T_*}\log \frac{\ep_*}{\theta^{\frac{\rho-1}{2-\rho}}\xi_{T_*}}.
$$
Since $\xi_T > \ep$ and $\xi_{T_*}\leq \ep_*$, we obtain
$$
\sum_{t=T_*}^T N^{(t)} =
T-T_*+1 \leq  1+ \left(\log\frac{1}{\rho-1}\right)^{-1}\left(\log\log \frac{\ep_*}{\theta^{\frac{\rho-1}{2-\rho}}\ep}-\log\log\frac{1}{\theta^{\frac{\rho-1}{2-\rho}}}\right).
$$
Consequently, $N$ is bounded as follows.
$$
N = \sum_{t=0}^{T_*-1} N^{(t)} + \sum_{t=T_*}^T N^{(t)} \leq N(\ep_*,\sigma_*,C) + 1+\left(\log\frac{1}{\rho-1}\right)^{-1}\left(\log\log \frac{\ep_*}{\theta^{\frac{\rho-1}{2-\rho}}\ep}-\log\log\frac{1}{\theta^{\frac{\rho-1}{2-\rho}}}\right).
$$

\fbox{(iv) Case $\rho=1$.}
We have
$
\outersigma{t} \geq
\sigma_*
$ for each $t=0,\ldots,T+1$, by Lemma~\ref{lem:rAdaAPG-sigma} (iii).
Moreover, Lemma~\ref{lem:rAdaAPG-base}~(vi) shows that, if $\outerXplus{t} \not\in X^*$, then we have
$$
1 \leq \frac{1}{\kappa}\left(\frac{L_f}{L_{\min}} + 1\right)\norm{g_{\outerM{t}}(\outerXinit{t})},
~~\text{i.e.,}~~
\norm{g_{\outerM{t}}(\outerXinit{t})} \geq \ep_*:=\kappa\left(\frac{L_f}{L_{\min}} + 1\right)^{-1}.
$$
In other words, the condition $\norm{g_{\outerM{t}}(\outerXinit{t})} < \ep_*$ must imply $\outerXplus{t}\in X^*$, which also yields $N^{(t)}=1$ and $\outerXinit{t+1}=\outerXplus{t} \in X^*$ by Proposition~\ref{prop:APG}~(v) (then the algorithm terminates at the ($t+1$)-th outer loop).
Therefore, we have
\begin{equation}\label{rho=1:gmap-ep*}
\norm{g_{\outerM{t}}(\outerXinit{t})} \geq \ep_*,\quad 0\leq t \leq T-1.
\end{equation}
\sloppy
Now we consider two cases. If $\norm{g_{\outerM{T}}(\outerXinit{T})} \geq \ep_*$ holds, then
combining with \eqref{pr:gradmap-lb-ep} yields ${\norm{g_{\outerM{T}}(\outerXinit{T})} \geq \max(\ep,\ep_*)}$, from which Lemma~\ref{lem:rAdaAPG-compl} with $T_*=T$ concludes
$$
N \leq N(\max(\ep,\ep_*),\sigma_*,C)
\quad\text{with}\quad
C=
\sqrt{\frac{1}{\sigma_*}}\left(1+\log_{1/\theta}\frac{\norm{g_{\outerM{0}}(\xinit)}}{\max(\ep,\ep_*)}\right).
$$
On the other case $\norm{g_{\outerM{T}}(\outerXinit{T})} < \ep_*$, we have $N^{(T)}=1$. Since \eqref{pr:gradmap-lb-ep} and \eqref{rho=1:gmap-ep*} implies ${\norm{g_{\outerM{T-1}}(\outerXinit{T-1})} \geq \max(\ep,\ep_*)}$, Lemma~\ref{lem:rAdaAPG-compl} with $T_*=T-1$ shows $\sum_{t=0}^{T-1}N^{(t)}\leq N(\max(\ep,\ep_*),\sigma_*,C)$. Hence,
$$
N = N^{(T)}+\sum_{t=0}^{T-1}N^{(t)} \leq 1 + N(\max(\ep,\ep_*),\sigma_*,C).
$$
This proves the desired bound on $N$.
To show the latter assertion of (iv), suppose $\ep<\ep_*$. Then, the output $\{\outerXinit{T+1},\outerM{T+1},\outerXplus{T+1}\}$ satisfies $\norm{g_{\outerM{T+1}}(\outerXinit{T+1})} \leq \ep < \ep_*$. Therefore, $\outerXplus{T+1}$ must be optimal.

The proof of Theorem~\ref{thm:rAdaAPG} is completed.
\end{proof}

\begin{proof}[\underline{Proof of Corollary~\ref{cor:rAdaAPG}}]
The function $N(\cdot,\cdot,\cdot)$ defined in \eqref{N-def} has the following expression.
$$
N(\ep,\sigma_*,C)
=
O\left( \log\frac{L_f+\sigma_*}{\sigma_*} \left[ \log\frac{\norm{g_0}}{\ep} + \log\frac{\sigmainit}{\sigma_*} + \sqrt{\frac{L_f}{\sigma_*}} + C\sqrt{L_f} \right] \right).
$$
By the choice \eqref{sigma0-choice-HEB} of $\sigmainit$, we can apply Corollary~\ref{cor:AdaAPG} and then the bound \eqref{simga0-bound} becomes
\begin{equation}\label{rAdaAPG-simgainit-bound}
\sigma(\outerXplus{0},\ep^{(0)}) \leq \sigmainit \leq \frac{2\gamma_{\rm inc}L_f}{1+\sqrt{2}\beta}.
\end{equation}
Then,  we have $\sigmainit \geq \bar\sigma$ since $\sigma(\outerXplus{0},\ep^{(0)}) \geq \bar\sigma$ holds by Lemma~\ref{lem:rAdaAPG-sigma} (ii).
Therefore, $\sigma_*$ in \eqref{simga*-def} becomes ${\sigma_*=\Theta(\bar\sigma)}$,
which also implies $\sigma_* = O(\sigmainit) = O(L_f)$ combined with \eqref{rAdaAPG-simgainit-bound}.
Applying the bounds ${\sigmainit=O(L_f)}$, $\sigma_*=O(L_f)$, and $\sigma_*=\Omega(\bar\sigma)$, we obtain
\begin{equation} \label{cor:rAdaAPG:N0}
N(\ep,\sigma_*,C)
=
O\left( \log\frac{L_f}{\bar\sigma} \left[ \log\frac{\norm{g_0}}{\ep} + \sqrt{\frac{L_f}{\bar\sigma}} + C\sqrt{L_f} \right] \right).
\end{equation}

If $\rho=2$, Theorem~\ref{thm:rAdaAPG} (i) implies $N \leq N(\ep,\sigma_*,C)$ with $C=O\left(\sqrt{\frac{1}{\bar\sigma}}\log\frac{\norm{g_0}}{\ep}\right)$.
In particular, \eqref{cor:rAdaAPG:N0} yields
\begin{equation}\label{cor:rAdaAPG:N1}
N(\ep,\sigma_*,C) = O\left(\sqrt{\frac{L_f}{\bar\sigma}}\log\frac{L_f}{\bar\sigma}\log\frac{\norm{g_0}}{\ep}\right).
\end{equation}
Since $\bar\sigma=\Theta(\kappa)$ by \eqref{bar-sigma-def}, we conclude the bound \eqref{rAdaAPG-compl-O} in the case $\rho=2$.

In the case $\rho=1$, using  Theorem~\ref{thm:rAdaAPG} (iv), the same argument as the case $\rho=2$ can be applied to obtain \eqref{cor:rAdaAPG:N1} replacing $\ep$ by $\max(\ep,\ep_*)$, where $\ep_*=\kappa(L_f/L_{\min}+1)^{-1}=\Omega(\kappa)$. Since $\bar\sigma=\Omega(\kappa^2/\Delta_0)$ by \eqref{bar-sigma-def}, we obtain the bound \eqref{rAdaAPG-compl-O} in the case $\rho=1$.

In the case $\rho \in ]1,2[$, we apply Theorem~\ref{thm:rAdaAPG} (iii) and the argument is similar to the previous cases. Therefore, the bound \eqref{rAdaAPG-compl-O} in this case can be obtained based on the estimate \eqref{cor:rAdaAPG:N1} replacing $\ep$ by $\max(\ep,\ep_*)$ and applying $\bar\sigma = \Omega(\kappa^{\frac{2}{\rho}}\Delta_0^{-\frac{2-\rho}{\rho}})$, $\ep_* = \Theta(\kappa^{\frac{1}{2-\rho}}L_f^{-\frac{\rho-1}{2-\rho}})$.

Finally, consider the case $\rho>2$.
By Lemma~\ref{lem:rAdaAPG-sigma} (ii) and \eqref{rAdaAPG-simgainit-bound}, remark that $\sigmainit \geq \sigma_*^{(0)} \geq \bar\sigma$ holds.
Then, Theorem~\ref{thm:rAdaAPG} (ii) implies $N \leq N(\ep,\sigma_*,C)$ with $C=O\left(\sqrt{\frac{1}{\sigmainit}} + \sqrt{\frac{1}{\bar\sigma}}\right)=O\left(\sqrt{\frac{1}{\bar\sigma}}\right)$.
Therefore, \eqref{cor:rAdaAPG:N0} applying $\bar\sigma = \Omega(\kappa^{\frac{1}{\rho-1}}\ep^{\frac{\rho-2}{\rho-1}})$ concludes the bound \eqref{rAdaAPG-compl-O} in the case $\rho>2$.
\end{proof}


\section{Conclusions}\label{sec:conclusion}

In this paper, we proposed two adaptive proximal gradient methods, Algorithms~\ref{alg:AdaAPG} and~\ref{alg:rAdaAPG}, for composite convex optimization problems.
These two algorithms are nearly optimal for the class of problems with smooth convex objective functions and for the class with an additional assumption, the H\"olderian error bound condition, respectively.
Without this additional assumption, it is unclear whether the latter algorithm also provides the near optimality.

An important question would be the establishment of optimal methods.
Remark that, under the gradient norm measure, a first-order method with optimal complexity for smooth convex functions was recently announced in \cite{Nes20}, namely, the complexity bound \eqref{low-compl-smooth} is tight (if $\dim \E$ is large enough).
Similarly, it is an important question whether we can improve the complexity \eqref{rAdaAPG-compl-O} to attain the lower bounds \eqref{low-compl-grad-HEB}.

The key idea of this work is the adaptive determination of the regularization parameter used to define the regularized objective function.
As proved in Theorem~\ref{thm:AdaAPG}, our method (Algorithm~\ref{alg:AdaAPG}) adapts the unknown and ideal regularization parameter defined in \eqref{sigma-th}.
This feature is also critical for the development of the restart scheme (Algorithm~\ref{alg:rAdaAPG}) to adapt the H\"olderian error bound condition. Basically, this adaption is obtained thanks to the relation between the ideal regularization parameter (in other words, the distance to the solution set) and the ``problem structure'' (cf. Lemma~\ref{lem:rAdaAPG-sigma}~(ii)).
This might suggest the possibility of dealing with this adaptive regularization approach under other kinds of problem structures.


{\small
\section*{Acknowledgements}
This work was partially supported by the Grant-in-Aid for Young Scientists (B) (17K12645) and the Grant-in-Aid for Scientific Research (C) (18K11178) from Japan Society for the Promotion of Science.
}



\end{document}